\documentclass[11pt,reqno]{amsart}
\usepackage{amsmath}
\usepackage{amssymb}
\usepackage{amsthm}
\usepackage{color}
\usepackage{eucal}
\usepackage{tikz}
\usepackage{gastex}

\usepackage{latexsym}
\usepackage{indentfirst}

\usetikzlibrary{decorations.pathmorphing}
\DeclareSymbolFont{rsfscript}{OMS}{rsfs}{m}{n}
\DeclareSymbolFontAlphabet{\mathrsfs}{rsfscript}

\def\softd{{\leavevmode\setbox1=\hbox{d}\hbox
            to 1.15\wd1{d\kern-0.2ex{\char039}\hss}}}    
\def\softl{l\kern-0.3ex\raise0.1ex\hbox{'}\kern-0.3ex}   
\numberwithin{equation}{section}
\newtheorem{prop}{Proposition}[section]

\newtheorem{lem}[prop]{Lemma}
\newtheorem{cor}[prop]{Corollary}

\theoremstyle{remark}

\def\Dc{\mathrel{\mathrsfs{D}}}

\def\Lc{\mathrel{\mathrsfs{L}}}
\def\Rc{\mathrel{\mathrsfs{R}}}
\def\A{\mathfrak{A}}

\def\B{\mathfrak{B}}

\def\T{\mathfrak{T}}
\def\S{\mathfrak{S}}
\def\R{\mathfrak{R}}
\def\L{\mathfrak{L}}

\def\wt{\widetilde}
\def\Ga{\Gamma}
\def\De{\Delta}
\def\Te{\Theta}
\def\co{\mathrm{c}}
\def\r{\mathrm{r}}
\def\l{\mathrm{l}}

\def\la{\lambda}
\def\al{\alpha}
\def\be{\beta}
\def\ga{\gamma}

\def\lv{\mathfrak{l}}
\def\rv{\mathfrak{r}}

\def\cb{\mathbf{c}}
\def\sk{\mathsf{sk}}
\def\Ls{\mathsf{L}}
\def\Rs{\mathsf{R}}
\def\Sub{\mathsf{Sub}}

\def\ol{\overline}
\def\wt{\widetilde}
\def\wh{\widehat}
\newcommand{\pseudo}{pseudosemilattice}
\newcommand{\spseudo}{strict pseudosemilattice}
\def\circledm{\protect\mathbin{\hbox
    {\protect$\bigcirc$\rlap{\kern-8.2pt\raise0pt\hbox
    {\protect$\mathtt{m}$}}}}}        
\def\smallcircledm{\protect\mathbin{\hbox
    {\protect$\bigcirc$\rlap{\kern-6.9pt\raise0pt\hbox
    {\protect$\mathtt{m}$}}}}}

\def\wr{\mathrel{{\le}_{\Rc}}}
\def\wl{\mathrel{{\le}_{\Lc}}}

\newcommand{\lra}{\longrightarrow}

\renewcommand{\iff}{if and only if}

\newcommand{\st}{such that}

\newcommand{\Vnki}{{\bf V}_{n,k,i}}

\newcommand{\Vnkj}{{\bf V}_{n,k,j}}
\newcommand{\Vnkjc}{{\bf V}^*_{n,k,j}}
\newcommand{\Vmlj}{{\bf V}_{m,l,j}}

\newcommand{\unkiv}{$u_{n,k,i}\approx v_{n,k,i}$}

\newcommand{\unkicv}{$u_{n,k,i}^*\approx v_{n,k,i}^*$}

\newcommand{\unki}{u_{n,k,i}}

\newcommand{\vnki}{v_{n,k,i}}

\newcommand{\anki}{\alpha_{n,k,i}}
\newcommand{\bnki}{\beta_{n,k,i}}
\newcommand{\ankj}{\alpha_{n,k,j}}
\newcommand{\bnkj}{\beta_{n,k,j}}
\newcommand{\amlj}{\alpha_{m,l,j}}
\newcommand{\bmlj}{\beta_{m,l,j}}

\newcommand{\nki}{(\anki,\bnki)}
\newcommand{\nkic}{(\anki^*,\bnki^*)}
\newcommand{\nkj}{(\ankj,\bnkj)}
\newcommand{\nkjc}{(\ankj^*,\bnkj^*)}
\newcommand{\mlj}{(\amlj,\bmlj)}
\newcommand{\mljc}{(\amlj^*,\bmlj^*)}

\title[]{A family of varieties of pseudosemilattices}
\author{Lu\'\i s Oliveira}
\address{Departamento de Matem\'atica Pura,
Faculdade de Ci\^encias da Universidade do Porto,
R. Campo Alegre, 687, 4169-007 Porto, Portugal}
\email{loliveir@fc.up.pt}

\begin{document}

\begin{abstract}
In \cite{l4}, a basis of identities $\{u_n\approx v_n\,\mid\; n\geq 2\}$ for the variety {\bf SPS} of all \spseudo s was determined. Each one of these identities $u_n\approx v_n$ has a peculiar 2-content $D_n$. In this paper we study the varieties of \pseudo s defined by sets of identities, all with 2-content the same $D_n$. We present here the family of all these varieties and show that each variety from this family is defined by a single identity also with 2-content $D_n$. This paper ends with the study of the inclusion relation between the varieties of this family.
\end{abstract}

\subjclass[2010]{08B15, 08B05, 08B20, 20M17}
\keywords{}

\maketitle

\section{Introduction}

A \emph{regular semigroup} is a semigroup $S$ for which every $x\in S$ has an $x'\in S$ \st\ $xx'x=x$. Thus $xx'$ and $x'x$ are idempotents in $S$ and we shall denote the set of all idempotents of $S$ by $E(S)$. Consider the following two binary relations on $E(S)$:  
\[e\,\wr f\,\Leftrightarrow\,e=fe\quad\mbox{ and }\quad e\,\wl f\,\Leftrightarrow\,e=ef;\] 
and set $\leq\,=\,\wr\cap\wl$. Then $\wr$ and $\wl$ are quasi-orders on $E(S)$, while $\leq$ is a partial order on $E(S)$. We shall denote by $(f]_{\Rc}$ the set of idempotents $e$ such that $e\,\wr f$. Similarly, we define $(f]_{\Lc}$ and $(f]_{\leq}\,$. We can introduce also the following two equivalence relations on $E(S)$:
\[e\Rc f \Leftrightarrow (e]_{\Rc}=(f]_{\Rc}\quad\mbox{ and }\quad e\Lc f \Leftrightarrow (e]_{\Lc}=(f]_{\Lc}\, ,\]
or equivalently, $\Rc\,=\,\wr\cap\mathrel{{\ge}_{\Rc}}\,$ and $\,\Lc\,=\,\wl\cap \mathrel{{\ge}_{\Lc}}$ for ${\ge}_{\Rc}$ and ${\ge}_{\Lc}$ the expected reverse relations corresponding to $\wr$ and $\wl$, respectively. Thus $\Rc\cap\Lc$ is just the identity relation. 

Any finite sequence $f_1, f_2,\cdots , f_m$ of alternately $\Rc$- or $\Lc$-equivalent idempotents of $S$ contains a subsequence $f_1=e_1,e_2,\cdots, e_n=f_m$ with the same property but where no two consecutive elements are equal. An \emph{$E$-chain} is then a finite sequence of alternately $\Rc$- or $\Lc$-equivalent idempotents with no two consecutive elements equal. Thus two idempotents are $(\Rc\vee\Lc)$-equivalent, and we shall say they are \emph{connected}, if there exists an $E$-chain starting at one of them and ending at the other one. The \emph{connected components} of $E(S)$ are the $(\Rc\vee\Lc)$-equivalence classes. 

The equivalence relations $\Rc$ and $\Lc$ are clearly the restriction to $E(S)$ of the homonymous Green's relations on the semigroup $S$. However, if we consider the Green's relation $\Dc$ ($=\,\Rc\vee\Lc$) on $S$, the relation $\Rc\vee\Lc$ on $E(S)$ defined above may not be the restriction of $\Dc$ to $E(S)$. In fact, the set of idempotents of a $\Dc$-class of $S$ is the (disjoint) union of several connected components of $E(S)$. Nevertheless, we can guarantee that the relation $\Rc\vee\Lc$ on $E(S)$ is the restriction of $\Dc$ if $S$ is generated by its idempotents.

A regular semigroup $S$ is \emph{locally inverse} if for every ordered pair $(e,f)$ of idempotents of $S$, $(e]_{\Rc}\cap (f]_{\Lc}=(g]_{\leq}$ for some $g\in E(S)$. The idempotent $g$ is clearly unique since $\leq$ is a partial order. Thus, we can define a new binary algebra $(E(S),\wedge)$ for every locally inverse semigroup $S$ by setting $e\wedge f=g$. The algebra $(E(S),\wedge)$ so obtained is called \emph{the \pseudo\ of idempotents} of $S$, and the quasi-orders $\wr$ and $\wl$ on $E(S)$ can be recovered from the binary operation $\wedge$ by setting:
\[e\wr f\Leftrightarrow f\wedge e=e\quad\mbox{ and }\quad e\wl f\Leftrightarrow e\wedge f=e\, .\]
The \pseudo s of idempotents of locally inverse semigroups are idempotent binary algebras, although they are often not semigroups themselves. Nevertheless, the class of all these binary algebras constitutes a variety (Nambooripad \cite{na1}) given by the identities:
\begin{itemize}
\item[$(PS1)$] $x\wedge x\approx x\;$;
\item[$(PS2)$] $(x\wedge y)\wedge (x\wedge z)\approx(x\wedge y)\wedge z\; $;
\item[$(PS3)$] $((x\wedge y)\wedge (x\wedge z))\wedge (x\wedge w)\approx (x\wedge y)\wedge ((x\wedge z)\wedge (x\wedge w))\;$;
\end{itemize}
together with the left-right duals (PS2') and (PS3') of (PS2) and (PS3), respectively. We shall denote this variety by $\bf PS$.

The structure of \pseudo s is related to the notion of semilattice. Every \pseudo s is the union of its maximal subsemilattices, and further, it is a homomorphic image of another \pseudo\ whose maximal subsemilattices are disjoint \cite{mp1}. A \pseudo\ with disjoint maximal subsemilattices can be described \cite{bmp, past} as the union of disjoint maximal semilattices $E_{i\lambda}$ for $(i,\lambda)\in I\times\Lambda$ such that:
\begin{itemize}
\item[$(i)$] if $x_{i\lambda}\in E_{i\lambda}$ and $x_{j\mu}\in E_{j\mu}$, then $x_{i\lambda}\wedge x_{j\mu} \in E_{i\mu}$; 
\item[$(ii)$] $\cup_{i\in I}E_{i\lambda}$ is a left normal band, that is, an idempotent semigroup satisfying $xyz\approx xzy$;
\item[$(iii)$] $\cup_{\lambda\in \Lambda}E_{i\lambda}$ is a right normal band, that is, an idempotent semigroup satisfying $xyz\approx yxz$.
\end{itemize}
We can say even more, every \pseudo\ divides an elementary \pseudo\ \cite{mp1}, that is, a \pseudo\ with disjoint maximal subsemilattices all isomorphic (the semilattices $E_{i\lambda}$ above are all isomorphic).

An e-variety of regular semigroups \cite{ha1, ks1} is a class of these algebras closed for taking homomorphic images, direct products and regular subsemigroups. The class $\bf LI$ of all locally inverse semigroups is an example of an e-variety. Nambooripad's result \cite{na1} was generalized by Auinger \cite{au4} who proved that the mapping
\[\varphi: {\mathcal{L}}_e({\bf LI})\lra {\mathcal{L}}({\bf PS}),\quad {\bf V}\longmapsto \{(E(S),\wedge)\,|\, S\in {\bf V}\}\]
is a well-defined complete homomorphism from the lattice ${\mathcal{L}}_e({\bf LI})$ of e-varie\-ties of locally inverse semigroups onto the lattice ${\mathcal{L}}({\bf PS})$ of varie\-ties of \pseudo s. Thus, any information about ${\mathcal{L}}({\bf PS})$ is useful to understand the structure of ${\mathcal{L}}_e({\bf LI})$ itself.  

A \spseudo\ is the \pseudo\ of idempotents of some [combinatorial] strict regular semigroup, that is, of some subdirect product of completely simple and/or 0-simple semigroups. The class {\bf SPS} of all \spseudo s is a variety, and in fact it is the smallest variety of \pseudo s containing algebras that are not semigroups. On the other hand, the largest variety of \pseudo s whose algebras are all semigroups is the variety $\bf NB$ of all normal bands. It is a well known fact that ${\bf NB}\subseteq {\bf SPS}$. The set of identities satisfied by all \spseudo s was characterized by Auinger \cite{au4}: an identity $u\approx v$ is satisfied by all \spseudo s \iff\ the words $u$ and $v$ have the same \emph{leftmost letter}, the same \emph{rightmost letter}, and the same \emph{2-content} (see section 2 for the definition of 2-content). In \cite{l4} a basis $\{u_n\approx v_n\,:\, n\geq 2\}$ of identities for $\bf SPS$ was determined. The 2-content $D_n$ of the word $u_n$ has a very peculiar nature. In this paper we shall study the varieties of \pseudo s defined by a single identity whose words have some $D_n$ ($n\geq 2$) as their 2-content. This will allow us to define a family of varieties of \pseudo s which will gives us some incite into the structure of the lattice ${\mathcal{L}}({\bf PS})$. 

The terminology introduced in \cite{l4} and the results obtained in that paper are crucial for the unwind of the present one. Thus we shall devote the next section to recall the concepts and results from \cite{l4}. We shall call an identity \emph{non-trivial} if there is a \pseudo\ which does not satisfies it. 

For each $n\geq 2$, a family 
\[\{u_{n,k,i}\approx v_{n,k,i},\,u_{n,k,j}^*\approx v_{n,k,j}^*\,\mid \,k\geq 1,\,1\leq i,j\leq 2n\mbox{ and } j \mbox{ odd}\}\] 
of non-trivial identities will be introduced in section 3 which generalizes the identity $u_n\approx v_n$ (in fact $u_n\approx v_n$ corresponds to $u_{n,1,1}\approx v_{n,1,1}$). We shall see that this family contains all the identities needed to describe varieties of \pseudo s defined by identities with 2-content $D_n$. The list of all such varieties will be presented in section 4 together with the description of the inclusion relation between them. Further, it is shown in that same section that every variety from that list is defined by a single identity with 2-content $D_n$. Finally, in the last section, we shall study the inclusion relation between varieties of \pseudo s defined by identities with 2-content $D_n$ and varieties of \pseudo s defined by identities with 2-content $D_m$ for $n\neq m$.

\section{Recalling concepts and results from \cite{l4}}

The present section is entirely devoted to recall concepts and results obtained and used in \cite{l4}. Thus most details are left to the reader to consult \cite{l4}. 

Let $X$ be a non-empty set whose elements shall be called \emph{letters}. We shall denote by $\T(X)$ the set of all finite \emph{downward} (connected) trees with $(i)$ a unique top vertex, \emph{the root}; $(ii)$ each non-root vertex has a unique predecessor (a vertex placed above it in the tree but connected to it by an edge); $(iii)$ each non-leaf vertex $a$ (vertex of degree at least 2) has exactly two successors (vertices placed below $a$ in the tree but connected to $a$ by an edge), one to the left of $a$ and one to the right of $a$; and $(iv)$ the \emph{leaves} (vertices of degree at most 1) are labeled by letters of $X$. Thus the leaves appear at the bottom of the trees of $\T(X)$. A non-root vertex of a tree of $\T(X)$ is called a left/right vertex if that vertex is placed to the left/right of its predecessor. Hence each non-root vertex is either a left vertex or a right vertex. The root is considered neither a left vertex, nor a right vertex. 

We shall denote by $(F_2(X),\wedge)$ the absolutely free binary algebra on $X$. Thus, the elements of $F_2(X)$ are well-formed words on the alphabet $X\cup\{(,),\wedge\}$. We can associate inductively a tree from $\T(X)$ to each word of $F_2(X)$ by setting $\Ga(x)=\underset{x}{\bullet}$ for each $x\in X$ and then letting
 $$\tikz[scale=1]{\draw(0,0) node{$\Ga(u\wedge v):=$};\draw(1.5,-0.5)node{$\Ga(u)$};\draw(2.5,-0.5)node{$\Ga(v)$};\draw(2,0.5)node{$\bullet$};
 \draw (1.5,-0.3)--(2,0.5);\draw (2,0.5)--(2.5,-0.3);}
$$
for $u,v\in F_2(X)$. The mapping $\Gamma:F_2(X)\rightarrow \T(X)$ so obtained is in fact a bijection (see \cite{l1}), and if we introduce the binary operation $\wedge$ on $\T(X)$ by setting $\Ga(u)\wedge\Ga(v)=\Ga(u\wedge v)$ for any $u,v\in F_2(X)$, we obtain a model $(\T(X),\wedge)$ for the absolutely free binary algebra on $X$. 

There is always some ambiguity when referring to subwords of a word $u$ since $u$ can have several distinct copies of the same word as a subword. For example, the letter $x$ occurs twice as a subword of $x\wedge x$. We shall use $\Ga(u)$ to avoid this ambiguity. Let $\Sub(u)$ denote the set of all subwords of $u$ including repetitions, that is, if $v$ is a subword of $u$ then $\Sub(u)$ as a distinct copy of $v$ for each occurrence of $v$ as a subword of $u$; and for each vertex $a$ of $\Ga(u)$, let $\Ga(u,a)$ denote the downward subtree of $\Ga(u)$ having $a$ as the top vertex. The graph $\Ga(u)$ captures all the subword structure of $u$ in the sense that there is a natural bijection $\eta_u:V(\Ga(u))\rightarrow \Sub(u)$, where $V(\Ga(u))$ is the set of vertices of $\Ga(u)$, such that if $\eta_u(a)=v$ then $\Ga(v)$ is (isomorphic to) $\Ga(u,a)$. Thus $\eta_u(a)=\Ga(u)$ for $a$ the root of $\Ga(u)$ and the leaves of $\Ga(u)$ are in bijection with the one-letter subwords of $u$. Further, if $b$ and $c$ are respectively the left and right successors of $a$ in $\Ga(u)$, then $\eta_u(a)$ is the subword $\eta_u(b)\wedge\eta_u(c)$ of $u$. We can use now the vertices of $\Ga(u)$ to pinpoint the concrete subword we are referring to and avoid in this way any possible ambiguity that may occur.

In the following we define some combinatorial invariants of the words $u\in F_2(X)$. Let $\l(u)$ and $\r(u)$ be respectively the \emph{leftmost} and \emph{rightmost} letter in $u$, and let $\co(u)$ be the \emph{content} of $u$, that is, the set of letters that occur in $u$. We define also the \emph{2-content} $\co_2(u)$ inductively by setting $\co_2(x)=\{(x,x)\}$ for each $x\in X$, and letting
\[\co_2(u\wedge v)=\co_2(u)\cup\{(\l(u),\r(v))\}\cup \co_2(v)\]
for $u,v\in F_2(X)$. These combinatorial invariants could be defined using instead $\Ga(u)$ in the obvious way. We shall use however $\Ga(u)$ to introduce two other combinatorial invariants: the \emph{left content} $\co_l(u)$ and the \emph{right content} $\co_r(u)$ of a word $u\in F_2(X)$ are respectively the labels of the left leaves and the labels of the right leaves of $\Ga(u)$. In particular $\co_l(x)=\co_r(x)=\emptyset$ for every $x\in X$ since the root vertex is considered neither a left vertex nor a right vertex.

There is another type of trees introduced in \cite{l4} with all vertices labeled by letters of $X$ that we shall recall now. So, let $\B'(X)$ be the set of all finite non-trivial (connected) trees $\ga$ whose vertices are labeled by letters of $X$ and with an ordered pair $(\lv_\ga,\rv_\ga)$ of distinguished vertices connected by an edge. The vertices $\lv_\ga$ and $\rv_\ga$ shall be called respectively the \emph{left root} and the \emph{right root} of $\ga$. We can now partition the vertices of $\ga$ into two disjoint sets accordingly to their distance to $\lv_\ga$ (or to $\rv_\ga$): the set $L_\ga$ of all vertices with even distance to $\lv_\ga$ (or with odd distance to $\rv_\ga$) and the set $R_\ga$ of all vertices with odd distance to $\lv_\ga$ (or with even distance to $\rv_\ga$). Thus $\lv_\ga\in L_\ga$ and $\rv_\ga\in R_\ga$, and any edge of $\ga$ connects a vertex of $L_\ga$ with a vertex of $R_\ga$. Thus, we shall consider $\ga$ always as a bipartite graph (in fact a bipartite tree), and if $a\in L_\ga$ and $b\in R_\ga$ are connected by an edge, then we shall represent that edge by the ordered pair $(a,b)$. Further, a vertex from $L_\ga$ shall be called a \emph{left vertex} of $\ga$ while a vertex from $R_\ga$ shall be called a \emph{right vertex} of $\ga$. 

Let $\B(X)=\B'(X)\cup\{\underset{x}{\bullet}\,:\,x\in X\}$ and define $\lv_\ga=\rv_\ga=\ga$ for $\ga=\underset{x}{\bullet}\,$. For each $\ga\in\B'(X)$ let ${}^L\ga$ coincide with $\ga$ but now only with the left root as a distinguished vertex. However, this unique distinguished vertex continues to be seen as a left vertex in ${}^L\ga$, that is, one still sees the left/right vertices of $\ga$ as left/right vertices of ${}^L\ga$. For each $\ga=\underset{x}{\bullet}$ let ${}^L\ga$ be just the graph $\ga$ but now seen as a bipartite graph with only one left vertex and no right vertices; the only vertex of ${}^L\ga$ is now distinguished as a left root. We define $\ga^R$ dually and introduce the binary operation $\sqcap$ on $\B(X)$ by setting
\[\al\sqcap\be={}^L\al\mathrel{\dot{\cup}}\{(\lv_\al,\rv_\be)\}\mathrel{\dot{\cup}}\be^R\]
for $\al,\be\in\B(X)$. In other words, we construct $\al\sqcap\be$ by taking the disjoint union of $\al$ and $\be$, adding the edge $(\lv_\al,\rv_\be)$, and setting $\lv_\al$ and $\rv_\be$ respectively as the left root and the right root of $\al\sqcap\be$.

The binary algebra $(\B(X),\sqcap)$ is generated as such by the set $\{\underset{x}{\bullet}\,:\,x\in X\}$. Thus, there is a unique surjective homomorphism $\Delta: F_2(X) \rightarrow\B(X)$ such that $\Delta(x)= \underset{x}{\bullet}=\Ga(x)\,$. There is a standard procedure to obtain $\Delta(u)$ from $\Gamma(u)$: first set the leaves of $\Ga(u)$ (together with their labels) as the vertices of $\Delta(u)$; then let $\lv_{\Delta(u)}$ and $\rv_{\Delta(u)}$ be the leftmost and the rightmost leaves of $\Ga(u)$ respectively; and finally add an edge to $\De(u)$ for each non-leaf vertex $a$ of $\Ga(u)$, which connects the leftmost and rightmost leaves of $\Ga(u,a)$. In particular, $L_{\Delta(u)}$ is the set of left leaves of $\Gamma(u)$ while $R_{\Delta(u)}$ is the set of right leaves of $\Gamma(u)$ for each $u\in F_2(X)\setminus\{X\}$.

It is convenient to see the procedure just described as a partial bijection $\chi_u:\Ga(u)\rightarrow\De(u)$ whose domain is the set of vertices of $\Ga(u)$ and whose image is the all graph $\De(u)$. Thus $\chi_u$ induces a label preserving bijection from the set of leaves of $\Ga(u)$ onto the set of vertices of $\De(u)$ such that if $a$ is a non-leaf vertex of $\Ga(u)$ and $b$ and $c$ are respectively the leftmost and rightmost leaves of $\Ga(u,a)$, then $\chi_u(a)$ is the edge $(\chi_u(b),\chi_u(c))$ of $\De(u)$. This procedure was described in \cite{l4} using a different but equivalent method involving the notion of \emph{contraction of subtrees}. We recommend the reader to consult \cite{l4} for more details including the illustration of a concrete example. In particular, it was pointed out in that paper that if all vertices of $\ga\in\B (X)$ have degree at most two, then there exists a unique $u\in F_2(X)$ such that $\ga=\De(u)$ (in general $\De$ is not injective). This last observation will be useful for this paper.

We shall denote by $\cb_a$ the label of a labeled vertex $a$. The combinatorial invariants $\l(u)$, $\r(u)$, $\co(u)$, $\co_2(u)$, $\co_l(u)$ and $\co_r(u)$ of a word $u\in F_2(X)$ can be described using $\De(u)$: $\l(u)$ is the label of $\lv_{\De(u)}$ while $\r(u)$ is the label of $\rv_{\De(u)}$; $\co(u)$ is the set of labels of all vertices of $\De(u)$ while $\co_l(u)$ and $\co_r(u)$ are respectively the set of labels of all vertices of $L_{\De(u)}$ and $R_{\De(u)}$; and $\co_2(u)$ is constituted by the pairs $(\cb_a,\cb_b)$ for all edges $(a,b)$ of $\De(u)$ together with the pairs $(\cb_a,\cb_a)$ for all vertices $a$ of $\De(u)$. We can define now the combinatorial invariants $\l(\ga)$, $\r(\ga)$, $\co(\ga)$, $\co_2(\ga)$, $\co_l(\ga)$ and $\co_r(\ga)$ for each $\ga\in\B(u)$ as expected: $\l(\ga)=\cb_{\lv_\ga}$, $\r(\ga)=\cb_{\rv_\ga}$, and so on.

Let $a$ be a vertex of $\Ga(u)$ and let $s=\eta_u(a)$, a subword of $u$. Denote by $u(s\rightarrow t)$ the word obtained by replacing in $u$ the subword $s=\eta_u(a)$ with some word $t\in F_2(X)$. Thus $\Ga(s)=\Ga(u,a)$ and $\Ga(u(s\rightarrow t))$ is obtained from $\Ga(u)$ by substituting $\Ga(t)$ for the downward subtree $\Ga(u,a)$. If $a$ is the root of $\Ga(u)$, then $s=u$ and $u(s\rightarrow t)=t$, whence $\De(s)=\De(u)$ and $\De(u(s\rightarrow t))=\De(t)$. Now, assume that $a$ is a left vertex of $\Ga(u)$ and let $b$ be the leftmost leaf of $\Ga(u,a)$. Then ${}^L\De(s)$ is just the connected subtree $\chi_u(\Ga(u,a))$ of $\De(u)$ with the vertex $\chi_u(b)$ as the distinguished vertex. Furthermore, $\chi_u(b)$ is the only vertex of $\chi_u(\Ga(u,a))$ connected by an edge to vertices outside $\chi_u(\Ga(u,a))$. Thus if $c$ denotes the left root of ${}^L\De(t)$, then $\De(u(s\rightarrow t))$ is obtained from $\De(u)$ by substituting ${}^L\De(t)$ for the subtree $\chi_u(\Ga(u,a))$ (and each edge $(\chi_u(b),d)\in\Ga(u)$ with $d\not\in \chi_u(\Ga(u,a))$ is replaced by the edge $(c,d)$). A similar situation occurs if $a$ is a right vertex of $\Ga(u)$.

If $\psi$ is an endomorphism of $F_2(X)$, then $u\psi$ is the word obtained by replacing in $u$ each one-letter subword $x$ with $x\psi$, or equivalently, $\Ga(u\psi)$ is the graph obtained from $\Ga(u)$ by replacing each leaf $a$ with $\Ga(\cb_a\psi)$. Thus $\De(u\psi)$ is the graph resulting from replacing in $\De(u)$ each left vertex $a$ by the left-rooted tree $\Ls(a,\psi)={}^L\De(\cb_a\psi)$ and each right vertex $b$ by the right-rooted tree $\Rs(b,\psi)= \De(\cb_b\psi)^R$ (all these graphs are assumed to be pairwise disjoint); and then setting $\lv_{\Ls(\lv_{\De(u)}, \psi)}$ as the left root and $\rv_{\Rs(\rv_{\De(u)},\psi)}$ as the right root. Alternatively, $\De(u\psi)$ can be obtained as follows: form the disjoint union
$$\left(\bigcup_{a\in L_{\De(u)}}\Ls(a,\psi)\right)\,\cup\,\left( \bigcup_{b\in R_{\De(u)}}\Rs(b,\psi)\right)$$
of all graphs $\Ls(a,\psi)$ and $\Rs(b,\psi)$ and add the edge $(\lv_{\Ls(a,\psi)},\rv_{\Rs(b,\psi)})$ for each edge $(a,b)$ of $\De(u)$. For $(a,b)=(\lv_{\De(u)},\rv_{\De(u)})$ this yields the connection between the distinguished vertices $\lv_{\Ls(\lv_{\De(u)},\psi)}$ and $\rv_{\Rs(\rv_{\De(u)}, \psi)}$ of $\De(u\psi)$. The following corollary has been already stated in \cite{l4} and is an immediate consequence of the previous description of $\De(u\psi)$.

\begin{cor} \label{deltafullyinvariant} For all $u,v\in F_2(X)$ and each endomorphism $\psi:F_2(X)\to F_2(X)$, if $\De(u)=\De(v)$ then $\De(u\psi)=\De(v\psi)$. In particular, $\mathrm{ker}\De$ is a fully invariant congruence on $F_2(X)$.
\end{cor} 

The \emph{skeleton} $\sk(u,\psi)$ of $\De(u\psi)$ is the subtree spanned by the set of vertices
$$\{\lv_{\Ls(a,\psi)}\mid a\in L_{\De(u)}\}\cup\{\rv_{\Rs(b,\psi)}\mid b\in R_{\De(u)}\}$$
(or spanned by all edges $(\lv_{\Ls(a,\psi)}, \rv_{\Rs(b,\psi)})$ for $(a,b)$ an edge in $\De(u)$). Further, we set $\lv_{L(\lv_{\De(u)},\psi)}$ and $\rv_{R(\rv_{\De(u)},\psi)}$ as the left and right roots of $\sk(u,\psi)$ respectively. Then $\sk(u,\psi)$ has the same graph structure as $\De(u)$ (including the same distinguished vertices) although with different vertex labels. To be more precise, the label of each $a\in L_{\De(u)}$ is changed from $\cb_a$ to $\l(\cb_a\psi)$ while the label of each $b\in R_{\De(u)}$ is changed from $\cb_b$ to $\r(\cb_b\psi)$. In case the left content of $u$ is disjoint from its right content, the skeleton $\sk(u,\psi)$ itself can be viewed as a graph of the form $\De(u\psi')$ for any endomorphism $\psi'$ satisfying $x\psi'=\l(x\psi)$ if $x\in \co_l(u)$ and $x\psi'=\r(x\psi)$ if $x\in \co_r(u)$. 

We need to make some conventions about the graphical representation of the bipartite graphs from $\B'(X)$. For each $\ga\in\B'(X)$ we shall arrange their vertices either in two columns (the left/right column representing the left/right vertices) or in two horizontal rows (the bottom/top row representing the left/right vertices); and we shall distinguish the left and right roots by especially representing the unique edge connecting these two vertices by a double line \tikz[baseline=0.15cm,scale=0.5]{\draw[double](0,0.5)--(1.5,0.5);}. When referring to one-vertex only distinguished graphs, we shall use an ``encircled bullet'' \tikz[baseline=-4pt,scale=0.3]{\coordinate  (1) at (0,0);
 \draw (1) node{$\bullet$}; \draw (1)circle (10pt);} to indicate the distinguished vertex.

Next we shall introduce two rules for changing the graphs of $\B'(X)$, but we need to define first the notion of thorn. A \emph{thorn} in a graph $\al\in\B'(X)$ is a pair $\{e,a\}$ consisting of a degree one vertex $a$ together with the edge $e$ having $a$ as one of its endpoints such that both $a$ and the other endpoint of $e$ have the same label. The thorn $\{e,a\}$ is called \emph{essential} if $a$ is one of the distinguished vertices of $\al$; otherwise it is called a \emph{non-essential thorn}. These notions of essential and non-essential thorns shall be considered also for one-vertex only distinguished graphs and for non-distinguished vertex graphs with the obvious adaptations. The two reduction rules are now introduced as follows:
\begin{enumerate}
\item[(i)] remove a non-essential thorn $\{e,a\}$ from $\al$. This rule may be visualized graphically as

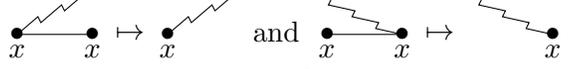
\centerline{\tikz[baseline=-0.1cm,scale=0.5]{
\coordinate[label=below:$x$] (1) at (0,0);
\coordinate[label=below:$x$] (2) at (2,0);
\coordinate[label=below:$x$] (3) at (4,0);
\draw (1) node{$\bullet$};
\draw (2) node{$\bullet$};
\draw (3) node{$\bullet$};
\draw (3,0) node{$\mapsto$};
\draw (1)--(2);
\draw[decorate,decoration=saw] (1)--(2,1);
\draw[decorate,decoration=saw] (3)--(6,1);
}
and \tikz[baseline=-0.1cm,scale=0.5]{
\coordinate[label=below:$x$] (1) at (0,0);
\coordinate[label=below:$x$] (2) at (2,0);
\coordinate[label=below:$x$] (3) at (6,0);
\draw (1) node{$\bullet$};
\draw (2) node{$\bullet$};
\draw (3) node{$\bullet$};
\draw (3,0) node{$\mapsto$};
\draw (1)--(2);
\draw[decorate,decoration=saw] (2)--(0,1);
\draw[decorate,decoration=saw] (3)--(4,1);
\filldraw (6.5,0)circle (0.5pt);}}
\item[(ii)] suppose that two edges $e$ and $f$ have a vertex in common and that the two other (distinct) vertices $a$ and $b$ have the same label; then identify the two edges $e$ and $f$ and the vertices $a$ and $b$ (and retain their label). If one of the merged vertices happens to be a distinguished one then so is the resulting vertex. Graphically, this rule may be visualized as

    
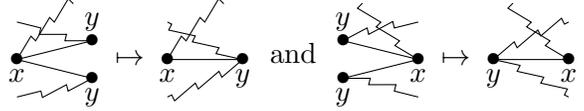
\centerline{
\tikz[baseline=0.2cm,scale=0.5]{\coordinate [label=below:$x$] (1) at (0,0.5);\coordinate [label=above:$y$] (2) at (2,1);
\coordinate[label=below:$y$] (3) at (2,0); \coordinate [label=below:$x$] (4) at (4,0.5); \coordinate [label=below:$y$] (5) at (6,0.5); \draw (2,1) -- (0,0.5) --(2,0); \draw (0,0.5) node{$\bullet$}; \draw (2,1) node{$\bullet$}; \draw (2,0) node{$\bullet$};\draw (3,0.5) node{$\mapsto$};\draw (4,0.5)--(6,0.5);\draw (4,0.5) node{$\bullet$}; \draw[decorate,decoration=saw](0,0.5)--(1.5,2);
\draw[decorate,decoration=saw](2,0)--(0,-0.5);
\draw[decorate,decoration=saw](2,1)--(0,1.5);
\draw[decorate,decoration=saw](4,0.5)--(5.5,2);
\draw[decorate,decoration=saw](6,0.5)--(4,-0.5);
\draw[decorate,decoration=saw](6,0.5)--(4,1.5);
\draw(6,0.5)node{$\bullet$};}
and \tikz[baseline=0.2cm,scale=0.5]{
\coordinate[label=above:$y$] (1) at (0,1);
\coordinate[label=below:$y$] (2) at (0,0);
\coordinate[label=below:$x$] (3) at (2,0.5);
\coordinate[label=below:$y$] (4) at (4,0.5);
\coordinate[label=below:$x$] (5) at (6,0.5);
\draw (1) node{$\bullet$};
\draw (2) node{$\bullet$};
\draw (3) node{$\bullet$};
\draw (4) node{$\bullet$};
\draw (5) node{$\bullet$};
\draw (3,0.5) node{$\mapsto$};
\draw (1)--(3)--(2); \draw (4)--(5);
\draw[decorate,decoration=saw](0,1)--(2,1.5);
\draw[decorate,decoration=saw](0,0)--(2,-0.5);
\draw[decorate,decoration=saw](2,0.5)--(0.5,2);
\draw[decorate,decoration=saw](4,0.5)--(6,1.5);
\draw[decorate,decoration=saw](4,0.5)--(6,-0.5);
\draw[decorate,decoration=saw](6,0.5)--(4.5,2);
\filldraw (6.5,0.5)circle (0.5pt);}
}
\end{enumerate}
Rule (i) is referred to as the \emph{deletion of a thorn} while rule (ii) is called an \emph{edge-folding}.

A graph $\al\in\B'(X)$ is called \emph{reduced} if none of the two rules above can be applied to it. If $\al\in\B'(X)$ is not reduced, then we can always obtain a reduced graph by applying the rules (i) and (ii) until no more reductions are possible. Of course, we can apply these reductions in many different orders. Nevertheless, we always get the same reduced graph from a given $\al$ independently of the order of reductions we choose to apply. We shall denote by $\ol{\al}$ the \emph{reduced form} of $\al$, that is, the unique reduced graph obtained from $\al$ by applying the rules (i) and (ii). We can however obtain $\ol{\al}$ by first carry out all possible edge-foldings and then carry out all possible thorn deletions (note that a thorn deletion does not produce any possible new edge-folding). Furthermore, the edge-folding reduced graph obtained form $\al$ is always the same independently of the order in which we apply the edge-folding reductions. We shall denote by $\wt{\al}$ the edge-folding reduced graph obtained from $\al$, and thus we can see the reduction from $\al$ into $\ol{\al}$ as the two step process $\al\rightarrow\wt{\al}\rightarrow \ol{\al}$. The first step of this process induces a natural graph homomorphism from $\al$ onto $\wt{\al}$ preserving the labels and the left and right roots. In the second step of this process, we can see $\ol{\al}$ both as a subgraph of $\wt{\al}$ (preserving the labels and the left and right roots) and as the graph obtained from $\wt{\al}$ by `contracting' to a vertex each non-essential thorn $\{a,e\}$ together with the other endpoint of $e$, say $b$, and keeping the label of $b$ (or $a$).

Let $\A(X)$ be the set of all reduced graphs from $\B'(X)$ and define a binary operation $\wedge$ on $\A(X)$ by setting
\[\al\wedge\be=\ol{\al\sqcap\be}\, .\]
Set also $\ol{\al}= $\tikz[baseline=2.25pt,scale=0.3]{\coordinate [label=below:$x$] (4) at (0,0.5); \coordinate [label=below:$x$] (5) at (2.5,0.5);\draw[double] (4) --(5);\draw (4) node{$\bullet$};\draw (5) node{$\bullet$};} for $\al=\underset{x}{\bullet}$ and let $\Te(u)=\ol{\De(u)}$. The mapping $F_2(X)\rightarrow\A(X), \;u\rightarrow\Te(u)$ is a surjective homomorphism, and in fact it is the canonical homomorphism which extends the mapping $x\rightarrow$\tikz[baseline=2.25pt, scale=0.3]{\coordinate [label=below:$x$] (4) at (0,0.5); \coordinate [label=below:$x$] (5) at (2.5,0.5);\draw[double] (4) --(5);\draw (4) node{$\bullet$};\draw (5) node{$\bullet$};} for each $x\in X$. It was proved in \cite{l4} that $\A(X)$ is a model for the free \pseudo\ on $X$ if we identify each $x\in X$ with \tikz[baseline=2.25pt,scale=0.3]{\coordinate [label=below:$x$] (4) at (0,0.5); \coordinate [label=below:$x$] (5) at (2.5,0.5);\draw[double] (4) --(5);\draw (4) node{$\bullet$};\draw (5) node{$\bullet$};}. For future reference we restate this result in the following proposition together with some more information about $\A(X)$ obtained in \cite{l4}:

\begin{prop}\label{modelA}
The binary algebra $(\A(X),\wedge)$ is a model for the free \pseudo\ on $X$ if we identify each $x\in X$ with \tikz[baseline=2.25pt,scale=0.3]{\coordinate [label=below:$x$] (4) at (0,0.5); \coordinate [label=below:$x$] (5) at (2.5,0.5);\draw[double] (4) --(5);\draw (4) node{$\bullet$};\draw (5) node{$\bullet$};}. Further, the maximal subsemilattices of $\A(X)$ are the sets
\[\S_{x,y}(X)=\{\al\in\A(X)\mid \l(\al)=x\mbox{ and }\r(\al)=y\}\]
for $x,y\in X$, while the maximal right normal subbands and the maximal left normal subbands are respectively
\[\R_x(X)=\{\al\in\A(X)\mid\l(\al)=x\}\]
and
\[\L_x(X)=\{\al\in\A(X)\mid\r(\al)=x\}\]
for $x\in X$.
\end{prop}

Given an endomorphism $\varphi$ of $F_2(X)$, we shall denote by $\ol{\varphi}$ the unique endomorphism of $\A(X)$ that makes the following diagram commute:

\centerline{\tikz{
\node (1) at (0,0) {$\A(X)$};
\node (2) at (3,0) {$\A(X)$};
\node (3) at (0,2) {$F_2(X)$};
\node (4) at (3,2) {$F_2(X)$};
\path[->>,font=\small] (3) edge node[left] {$\Te$} (1);
\path[->>,font=\small] (4) edge node[right] {$\Te$} (2);
\path[->,font=\small] (1) edge node[below] {$\ol{\varphi}$} (2);
\path[->,font=\small] (3) edge node[above] {$\varphi$} (4);
}}

\noindent Thus $(\Te(u))\ol{\varphi}=\Te(u\varphi)=\ol{\De(u\varphi)}$ for any $u\in F_2(X)$, and in particular $($\tikz[baseline=2.25pt,scale=0.3]{\coordinate [label=left:$x$] (4) at (0,0.5); \coordinate [label=right:$x$] (5) at (2.5,0.5);\draw[double] (4) --(5);\draw (4) node{$\bullet$};\draw (5) node{$\bullet$};}$)\ol{\varphi}=$\tikz[baseline=2.25pt,scale=0.3]{\coordinate [label=left:$y$] (4) at (0,0.5); \coordinate [label=right:$y$] (5) at (2.5,0.5);\draw[double] (4) --(5);\draw (4) node{$\bullet$};\draw (5) node{$\bullet$};} if $x\varphi=y\in X$. Hence, if $X\varphi\subseteq X$, then we can obtain $\al\ol{\varphi}$ for $\al\in\A(X)$ by first setting $\be$ to be the graph $\al$ but with each label $x$ changed to $x\varphi$, and then reducing $\be$; that is, $\al\ol{\varphi}=\ol{\be}$. Conversely, given an endomorphism $\psi$ of $\A(X)$, we can construct an endomorphism $\varphi$ of $F_2(X)$ such that $\psi=\ol{\varphi}$, namely by setting for each $x\in X$, $x\varphi\in F_2(X)$ such that $\De(x\varphi)=($\tikz[baseline=2.25pt,scale=0.3]{\coordinate [label=left:$x$] (4) at (0,0.5); \coordinate [label=right:$x$] (5) at (2.5,0.5);\draw[double] (4) --(5);\draw (4) node{$\bullet$};\draw (5) node{$\bullet$};}$)\psi$. We should alert the reader that we shall jump very often between endomorphisms $\varphi$ of $F_2(X)$ and their corresponding endomorphisms $\ol{\varphi}$ of $\A(X)$ without further notice. 
 
Let $\al\in\A(X)$ and set $\wh\al=\underset{x}{\bullet}$ if $\al=$ \tikz[baseline=2.25pt,scale=0.3]{\coordinate [label=below:$x$] (4) at (0,0.5); \coordinate [label=below:$x$] (5) at (2.5,0.5);\draw[double] (4) --(5);\draw (4) node{$\bullet$};\draw (5) node{$\bullet$};} for some $x\in X$; otherwise, let $\wh\al$ be the (non-rooted) bipartite tree obtained from $\al$ by un-marking the distinguished roots and removing the existing thorns (that is, the thorns that were essential in $\al$). However, in this last case, the vertices of $\wh\al$ continue to be divided in left and right vertices as they were in $\al$. Define also
\[{}^l\al=\begin{cases} {}^L\al\setminus\{(\lv_\al,\rv_\al),\rv_\al\} &\text{ if }\{(\lv_\al,\rv_\al),\rv_\al\}\text{ is a thorn }\\ {}^L\al &\text{ otherwise}\end{cases}
\]
and $\al^r$ dually. If $\al=$\tikz[baseline=2.25pt,scale=0.3]{\coordinate [label=below:$x$] (4) at (0,0.5); \coordinate [label=below:$x$] (5) at (2.5,0.5);\draw[double] (4) --(5);\draw (4) node{$\bullet$};\draw (5) node{$\bullet$};} then ${}^l\al$ is the singleton graph $\underset{x}{\bullet}$ considered as a bipartite graph with one (distinguished) left vertex and no right vertex; the dual is assumed for $\al^r$. The next result compiles the main results of subsection 3.3 of \cite{l4}.

\begin{prop}\label{relA}
Let $\al,\be\in\A(X)$; then
\begin{enumerate}
\item $\be\wr\al$ if an only if ${}^l\al$ is a left-rooted subtree of ${}^l\be$;
\item $\be\wl\al$ if and only if $\al^r$ is a right-rooted subtree of $\be^r$;
\item $\be\le\al$ if and only if $\al$ is a bi-rooted subtree of $\be$;
\item $\al\Rc\be$ if and only if ${}^l\al={}^l\be$;
\item $\al\Lc\be$ if and only if $\al^r=\be^r$;
\item $\al \mathrel{(\Rc\vee\Lc)}\be$ if and only if $\wh\al=\wh\be$.
\end{enumerate}
\end{prop} 

\noindent In particular, if $\al$ covers $\be$ (that is, $\be\leq\al$ and if $\be\leq\ga\leq\al$ for some $\ga$ then $\be=\ga$ or $\al=\ga$) and $\co_l(\be)\cap\co_r(\be)=\emptyset$, then $\be$ has exactly one more vertex (and one more edge) than $\al$.

In this paper we shall talk about identities in the context of varieties of \pseudo s. For example, when we say that two identities are equivalent, we mean that the \pseudo s that satisfy one of them are the same that satisfy the other one. For each variety $\bf V$ of \pseudo s there exists a fully invariant congruence $\rho_{\bf V}(X)$ on $\A(X)$ such that $\A(X)/\rho_{\bf V}(X)$ is the relatively free algebra on $X$ for the variety $\bf V$. Then, a set of identities $I$ is a basis of identities for $\bf V$ \iff\ $\{(\Te(u),\Te(v))\mid u\approx v\in I\}$ generates $\rho_{\bf V}(X)$ as a fully invariant congruence for $X$ a countably infinite set. We shall write only $\rho_{\bf V}$ instead of $\rho_{\bf V}(X)$ when we are considering the set $X$ to be countably infinite. We shall say that a binary relation $\sigma$ on $\A(X)$ is a \emph{consequence} of another binary relation $\tau$ if $\sigma$ is contained in the fully invariant congruence generated by $\tau$ (or equivalently, if the variety defined by the identities induced by $\sigma$ contains the variety defined by the identities induced by $\tau$). Two binary relations on $\A(X)$ are said to be \emph{equivalent} if they generate the same fully invariant congruences (or equivalently, if the corresponding identities define the same variety of \pseudo s). 

By \cite{au1} an identity $u\approx v$ is satisfied by all \spseudo s \iff\ \[(\l(u),\co_2(u),\r(u))=(\l(v),\co_2(v),\r(v))\,.\]
Thus $(\al,\be)\in\rho_{\bf SPS}$ \iff\ $\co_2(\al)=\co_2(\be)$ and $\al$ and $\be$ belong to the same maximal subsemilattice of $\A(X)$. A pair $(\al,\be)$ of elements of $\A(X)$ is called \emph{elementary} if
\begin{enumerate}
\item $(\l(\al),\co_2(\al),\r(\al))=(\l(\be),\co_2(\be), \r(\be))$,
\item $\co_l(\be)\cap\co_r(\be)=\emptyset$ (or equivalently $\co_l(\al)\cap\co_r(\al)=\emptyset$),
\item $\al$ covers $\be$,
\item and the unique degree $1$ vertex in $\be\setminus\al$ is adjacent to a distinguished vertex of $\be$.
\end{enumerate}
Thus elementary pairs induce non-trivial identities satisfied by all \spseudo s. In fact, for each elementary pair $(\al,\be)$ we can always find an identity $u\approx v$ such that $(\De(u),\De(v))=(\al,\be)$ and $v$ is obtained from $u$ by replacing either the first letter of $u$ or the last letter of $u$, say $x$, respectively with $x\wedge y$ or $y\wedge x$ for some $y\in X$. We compile in the following result some conclusions obtained in \cite{l4} although not all of them are explicitly stated their.

\begin{prop}\label{EquivEl}
If $(\al,\be)$ is a non-trivial pair of $\rho_{\bf SPS}$, then $(\al,\be)$ is equivalent to a finite set $I$ of elementary pairs whose graphs all belong to the same maximal subsemilattice of $\A(X)$ and have the same 2-content. Further, if $\al$ has disjoint left and right contents, then $(\l(\al_1), \co_2(\al_1),\r(\al_1))=(\l(\al),\co_2(\al),\r(\al))$ for each $(\al_1,\be_1)\in I$; otherwise, we can always say that $|\co(\al_1)|\leq 2|\co(\al)|$.
\end{prop} 

It was shown also in \cite{l4} that we do not need to be too rigorous about the last condition in the definition of elementary pair.

\begin{prop}\label{DistVert}
Let $(\al,\be)$ be an elementary pair and let $(a,b)$ be an edge of $\al$. Let $\al_1$ and $\be_1$ be respectively the graphs $\al$ and $\be$ but now with the vertices $a$ and $b$ as the distinguished vertices. Then $(\al,\be)$ and $(\al_1,\be_1)$ are equivalent.
\end{prop}

For each integer $n\geq 2$ let $x_1,\ldots,x_{2n}$ be distinct letters from $X$. We designate by $\al_n$ and $\be_n$ the following graphs from $\A(X)$:

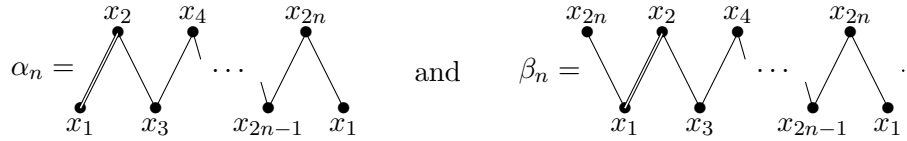
\centerline{
\tikz[scale=0.5]{\draw(-1,1) node{$\al_n=$};
\coordinate[label=below:$x_1$] (1) at (0,0);
\coordinate[label=above:$x_2$] (2) at (1,2);
\coordinate[label=below:$x_3$] (3) at (2,0);
\coordinate[label=above:$x_4$] (4) at (3,2);
\coordinate[label=below:$x_{2n-1}$] (5) at (5,0);
\coordinate[label=above:$x_{2n}$] (6) at (6,2);
\coordinate[label=below:$x_{1}$] (7) at (7,0);
\draw (4,1)node{$\dots$};
\draw (1) node {$\bullet$};
\draw (2) node {$\bullet$};
\draw (3) node {$\bullet$};
\draw (4) node {$\bullet$};
\draw (5) node {$\bullet$};
\draw (6) node {$\bullet$};
\draw (7) node {$\bullet$};
\draw (4)--(3.2,1.3);
\draw (4.8,0.7)--(5);
\draw [double] (1)--(2);
\draw (2)--(3)--(4);
\draw (5)--(6)--(7);
\draw(9.5,1) node{ and};}
\quad
\tikz[scale=0.5]{\draw(-2,1) node{$\be_n=$};
\coordinate[label=below:$x_1$] (1) at (0,0);
\coordinate[label=above:$x_2$] (2) at (1,2);
\coordinate[label=below:$x_3$] (3) at (2,0);
\coordinate[label=above:$x_4$] (4) at (3,2);
\coordinate[label=below:$x_{2n-1}$] (5) at (5,0);
\coordinate[label=above:$x_{2n}$] (6) at (6,2);
\coordinate[label=below:$x_{1}$] (7) at (7,0);
\coordinate[label=above:$x_{2n}$] (8) at (-1,2);
\draw (4,1)node{$\dots$};
\draw (1) node {$\bullet$};
\draw (2) node {$\bullet$};
\draw (3) node {$\bullet$};
\draw (4) node {$\bullet$};
\draw (5) node {$\bullet$};
\draw (6) node {$\bullet$};
\draw (7) node {$\bullet$};
\draw (8) node {$\bullet$};
\draw (4)--(3.2,1.3);
\draw (4.8,0.7)--(5);
\draw (8)--(1);
\draw [double] (1)--(2);
\draw (2)--(3)--(4);
\draw (5)--(6)--(7);
\filldraw (7.4,1)circle (0.5pt);}}

\noindent The pairs $(\al_n,\be_n)$ for $n\geq 2$ are obviously elementary pairs, and we shall designate by $D_n$ the 2-content of $\al_n$ (or of $\be_n$). Thus, $D_n$ is constituted by
\[\{(x_{2i-1},x_{2i}),\,(x_{2i+1},x_{2i})\mid 1\leq i<n\}\cup\{(x_1,x_{2n}),\,(x_{2n-1},x_{2n})\}\]
together with $\{(x_i,x_i)\mid 1\leq i\leq 2n\}$.
Since all vertices of both $\al_n$ and $\be_n$ have degree at most 2, there are unique words $u_n$ and $v_n$ such that $\De(u_n)=\al_n$ and $\De(v_n)=\be_n$. The main result of section 4 of \cite{l4} states that the set of all identities $u_n\approx v_n$, $n\geq 2$, is a basis of identities for the variety of all \spseudo s.

\begin{prop}
The set $\{u_n\approx v_n\mid n\geq 2\}$ is a basis of identities for the variety of all \spseudo s, or equivalently, the set $\{(\al_n,\be_n)\mid n\geq 2\}$ generates $\rho_{\bf SPS}$ as a fully invariant congruence for $X$ a countably infinite set.
\end{prop}

We shall end this section by associating another graph to each $\ga\in\B'(X)$. Let $\check{\ga}$ be the graph underlying $\ga$, that is, $\check{\ga}$ is just the graph $\ga$ but now with no labels on the vertices and with no distinguished vertices. However, we continue to see $\check{\ga}$ as a bipartite graph with `left' and `right' vertices as in $\ga$, that is, a left vertex of $\ga$ continues to be considered a left vertex in $\check{\ga}$ and the same occurs for the right vertices.

\section{Pairs of $\rho_{\bf SPS}$ with 2-content $D_n$}

Fix $n\geq 2$ for the next two sections. In this section we begin the study of the pairs $(\al,\be)\in\rho_{\bf SPS}$ such that $\co_2(\al)=D_n$, or in other words, we shall begin to investigate the identities $u\approx v$ satisfied by all \spseudo s and whose 2-content is $D_n$ (that is, $\co_2(u)=D_n$). We shall prepare here the ground for the next section where we determine all varieties of \pseudo s defined by identities satisfied by all \spseudo s and whose 2-content is $D_n$.

We first reinforce that we are fixing $n\geq 2$. Note that for $n=1$, the natural definition for $D_1$ would be $D_1=\{(x_1,x_2),\,(x_1,x_1),\,(x_2,x_2)\}$ and the only three graphs of $\A(X)$ with 2-content $D_1$ are \tikz[baseline=1pt,scale=0.3]{\coordinate [label=left:$x_1$] (4) at (0,0.5); \coordinate [label=right:$x_2$] (5) at (2.5,0.5);\draw[double] (4) --(5);\draw (4) node{$\bullet$};\draw (5) node{$\bullet$};},
\tikz[baseline=0pt,scale=0.3]{\coordinate [label=left:$x_1$] (1) at (0,0.5);\coordinate [label=right:$x_2$] (2) at (2,1);
\coordinate[label=right:$x_1$] (3) at (2,0);  \draw (2,1) -- (0,0.5);\draw[double] (0,0.5) --(2,0); \draw (0,0.5) node{$\bullet$}; \draw (2,1) node{$\bullet$}; \draw (2,0) node{$\bullet$};} and \tikz[baseline=0pt,scale=0.3]{\coordinate [label=left:$x_2$] (1) at (0,0);\coordinate [label=left:$x_1$] (2) at (0,1);
\coordinate[label=right:$x_2$] (3) at (2,0.5);  \draw (0,1) -- (2,0.5);\draw[double] (0,0) --(2,0.5); \draw (2,0.5) node{$\bullet$}; \draw (0,1) node{$\bullet$}; \draw (0,0) node{$\bullet$};}; no non-trivial pair formed by these graphs exists in $\rho_{\bf SPS}$. If $\al\in\A(X)$ has 2-content $D_n$ and has no essential thorn, then the underlying graph $\check{\al}$ is a `zig-zag segment', that is, has one of the following four configurations:

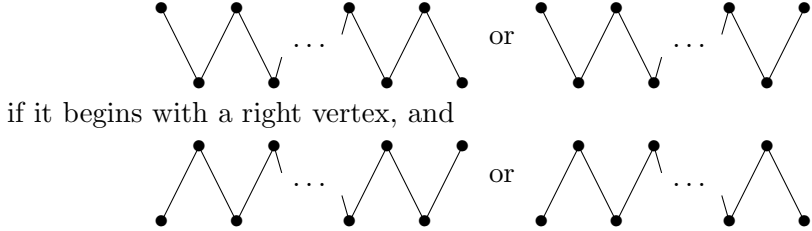
\centerline{
\tikz[baseline=15pt,scale=0.5]{
\coordinate (1) at (0,2);
\coordinate (2) at (1,0);
\coordinate (3) at (2,2);
\coordinate (4) at (3,0);
\coordinate (5) at (5,2);
\coordinate (6) at (6,0);
\coordinate (7) at (7,2);
\coordinate (8) at (8,0);
\draw (4,1)node{$\dots$};
\draw (1) node {$\bullet$};
\draw (2) node {$\bullet$};
\draw (3) node {$\bullet$};
\draw (4) node {$\bullet$};
\draw (5) node {$\bullet$};
\draw (6) node {$\bullet$};
\draw (7) node {$\bullet$};
\draw (8) node {$\bullet$};
\draw (4)--(3.2,0.7);
\draw (4.8,1.3)--(5);
\draw (1)--(2)--(3)--(4);
\draw (5)--(6)--(7)--(8);}
or
\tikz[baseline=15pt,scale=0.5]{
\coordinate (1) at (0,2);
\coordinate (2) at (1,0);
\coordinate (3) at (2,2);
\coordinate (4) at (3,0);
\coordinate (5) at (5,2);
\coordinate (6) at (6,0);
\coordinate (7) at (7,2);
\draw (4,1)node{$\dots$};
\draw (1) node {$\bullet$};
\draw (2) node {$\bullet$};
\draw (3) node {$\bullet$};
\draw (4) node {$\bullet$};
\draw (5) node {$\bullet$};
\draw (6) node {$\bullet$};
\draw (7) node {$\bullet$};
\draw (4)--(3.2,0.7);
\draw (4.8,1.3)--(5);
\draw (1)--(2);
\draw (2)--(3);
\draw (3)--(4);
\draw (5)--(6)--(7);}
}
{\parindent=0pt if it begins with a right vertex}, and

\centerline{
\tikz[baseline=15pt,scale=0.5]{
\coordinate (1) at (0,0);
\coordinate (2) at (1,2);
\coordinate (3) at (2,0);
\coordinate (4) at (3,2);
\coordinate (5) at (5,0);
\coordinate (6) at (6,2);
\coordinate (7) at (7,0);
\coordinate (8) at (8,2);
\draw (4,1)node{$\dots$};
\draw (1) node {$\bullet$};
\draw (2) node {$\bullet$};
\draw (3) node {$\bullet$};
\draw (4) node {$\bullet$};
\draw (5) node {$\bullet$};
\draw (6) node {$\bullet$};
\draw (7) node {$\bullet$};
\draw (8) node {$\bullet$};
\draw (4)--(3.2,1.3);
\draw (4.8,0.7)--(5);
\draw (1)--(2);
\draw (2)--(3);
\draw (3)--(4);
\draw (5)--(6)--(7)--(8);}
or
\tikz[baseline=15pt,scale=0.5]{
\coordinate (1) at (0,0);
\coordinate (2) at (1,2);
\coordinate (3) at (2,0);
\coordinate (4) at (3,2);
\coordinate (5) at (5,0);
\coordinate (6) at (6,2);
\coordinate (7) at (7,0);
\draw (4,1)node{$\dots$};
\draw (1) node {$\bullet$};
\draw (2) node {$\bullet$};
\draw (3) node {$\bullet$};
\draw (4) node {$\bullet$};
\draw (5) node {$\bullet$};
\draw (6) node {$\bullet$};
\draw (7) node {$\bullet$};
\draw (4)--(3.2,1.3);
\draw (4.8,0.7)--(5);
\draw (1)--(2);
\draw (2)--(3);
\draw (3)--(4);
\draw (5)--(6)--(7);}
}
{\parindent=0pt  if it begins} with a left vertex (the first graph of each row are in fact isomorphic). Moreover, $\al$ becomes completely determined once we identify in $\check{\al}$ the distinguished vertices and their labels. In other words, once we know which vertices are the distinguished vertices and we know their labels, the labels of all other vertices of $\al$ become fixed and easily determined. Furthermore, since all vertices of $\al$ have degree at most 2, there exists a unique word $u\in F_2(X)$ such that $\De(u)=\al$.

\begin{lem}\label{uniqueword}
Let $\al\in\A(X)$ be such that $\co_2(\al)=D_n$. If $\al$ has no essential thorn, then there exists a unique word $u\in F_2(X)$ such that $\De(u)=\al$. If $a$ and $b$ are two distinct and non-connected vertices of $\al$ with the same label, then the geodesic path from $a$ to $b$ has either $2nk+2$ vertices if $a$ or $b$ belong to some essential thorn of $\al$, or $2nk+1$ otherwise, for some $k\geq 1$.
\end{lem}

\begin{proof}
The first part has been observed already above and follows from the fact that all vertices of $\al$ have degree at most two. For the second part, assume first that $a$ and $b$ do not belong to some essential thorn. Let $a=a_0,a_1,\cdots, a_m=b$ be the sequence of vertices in the geodesic path from $a$ to $b$ (we are assuming that two consecutive vertices are connected by an edge). Let $\cb_{a_0}=x_i$. Then either $\cb_{a_1}=x_{i+1}$ (or $\cb_{a_1}=x_1$ if $i=2n$), or $\cb_{a_1}=x_{i-1}$ (or $\cb_{a_1}=x_{2n}$ if $i=1$). We shall assume that $\cb_{a_1}=x_{i+1}$ and prove only this case since the other one is similar. Now, the labels of all other vertices $a_j$ are fixed recursively by the rules: (i) if $\cb_{a_j}=x_{2n}$ then $\cb_{a_{j+1}}=x_1$, and (ii) if $\cb_{a_j}=x_l$ with $l\neq 2n$ then $\cb_{a_{j+1}}=x_{l+1}$. We can now easily observe that $\cb_{a_j}=\cb_{a_0}$ \iff\ $j=2nk$ for some $k\geq 0$. In particular $m=2nk$ for some $k\geq 1$ and there are $2nk+1$ vertices in the geodesic path from $a=a_0$ to $b=a_m$. 

Finally, assume that $a$ belongs to some essential thorn. Then $a$ is a distinguished vertex of $\al$ with degree 1 and if $c$ is the other distinguished vertex of $\al$ then $c$ and $a$ have the same label. Thus $c$ and $b$ are two vertices of $\al$ with the same label and not belonging to some essential thorn, whence there are $2nk+1$ vertices in the geodesic path from $c$ to $b$ for some $k\geq 1$. It is clear now that there are $2nk+2$ vertices in the geodesic path from $a$ to $b$.
\end{proof}

Let $(\al,\be)$ be an elementary pair with $\co_2(\al)=D_n$. Then both $\al$ and $\be$ have no essential thorns and their underlying graphs have each one of the four configurations depicted above. In fact, since $(\al,\be)$ is an elementary pair, we can assume that the distinguished vertices of $\al$ are the two first vertices (going from the left to the right); and then $\be$ is obtained from $\al$ by adding a vertex $a$ on the left side and connecting it by an edge to the leftmost vertex of $\al$ (if the leftmost vertex of $\al$ is a left vertex, then $a$ is a right vertex; otherwise $a$ is a left vertex). 

Let $\l(\al)=x_i$ ($i$ odd) and consider the permutation 
\[\sigma=\left(\begin{array}{cccccccc}
x_1 & x_2&\cdots &x_{i-1}&x_i&x_{i+1}&\cdots &x_{2n}\\
x_{2n+2-i}&x_{2n+3-i}&\cdots &x_{2n}&x_1&x_2&\cdots &x_{2n+1-i}
\end{array}\right)\, .\]
Extend $\sigma$ to an automorphism $\varphi$ of $F_2(X)$ by setting $x\varphi=x$ for any $x\not\in\{x_1,\cdots, x_{2n}\}$. Then, for each $\ga\in\A(X)$, $\ga\ol{\varphi}$ is just the graph $\ga$ but with each label $x_i$ changed to $\sigma(x_i)$. Hence $(\al\ol{\varphi},\be\ol{\varphi})$ is another elementary pair equivalent to $(\al,\be)$ and with $\co_2(\al\ol{\varphi})=D_n$. Further $\l(\al\ol{\varphi})=x_1$ and so $\r(\al\ol{\varphi})=x_2$ or $\r(\al\ol{\varphi})=x_{2n}$. If $\r(\al\ol{\varphi})=x_{2n}$ then consider the automorphism 
\[x\psi=\left\{\begin{array}{cl} x_{2n+2-i} & \mbox{ if } x=x_i \mbox{ and } 1<i\leq 2n\\ [.2cm] 
x & \mbox{ otherwise}
\end{array}\right.\]
of $F_2(X)$ and observe that $((\al)\ol{\varphi}\circ\ol{\psi},(\be)\ol{\varphi}\circ\ol{\psi})$ is once more an elementary pair equivalent to $(\al,\be)$, with $\co_2((\al)\ol{\varphi}\circ \ol{\psi})=D_n$, but now with distinguished vertices labeled by $x_1$ and $x_2$. 

Summing up the conclusions of the previous paragraph, we can assume that the labels of the distinguished vertices of the graphs of an elementary pair $(\al,\be)$ with $\co_2(\al)=D_n$ are always $x_1$ and $x_2$, and so $(\al,\be)$ becomes completely determined once we know the number of vertices of $\al$ and if the vertex in $\be\setminus\al$ is a left vertex or a right vertex. Assume that $\al$ has $m$ vertices and write $m=2nk+i$ for $k\geq 1$ and $1\leq i\leq 2n$ (there is obviously only one choice for $k$ and $i$). Assume further that the vertex in $\be\setminus\al$ is a right vertex and let

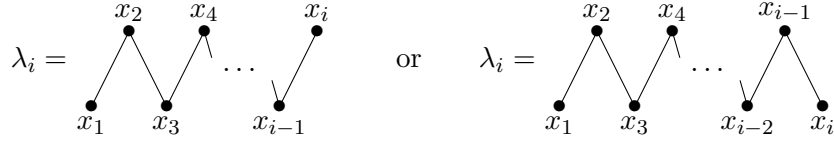
\centerline{
$\la_i=$\tikz[baseline=15pt,scale=0.5]{
\coordinate[label=below:$x_1$] (1) at (0,0);
\coordinate[label=above:$x_2$] (2) at (1,2);
\coordinate[label=below:$x_3$] (3) at (2,0);
\coordinate[label=above:$x_4$] (4) at (3,2);
\coordinate[label=below:$x_{i-1}$] (5) at (5,0);
\coordinate[label=above:$x_i$] (6) at (6,2);
\draw (4,1)node{$\dots$};
\draw (1) node {$\bullet$};
\draw (2) node {$\bullet$};
\draw (3) node {$\bullet$};
\draw (4) node {$\bullet$};
\draw (5) node {$\bullet$};
\draw (6) node {$\bullet$};
\draw (4)--(3.2,1.3);
\draw (4.8,0.7)--(5);
\draw (1)--(2);
\draw (2)--(3);
\draw (3)--(4);
\draw (5)--(6);}
\hspace*{.5cm} or \hspace*{.5cm}
$\la_i=$\tikz[baseline=15pt,scale=0.5]{
\coordinate[label=below:$x_1$] (1) at (0,0);
\coordinate[label=above:$x_2$] (2) at (1,2);
\coordinate[label=below:$x_3$] (3) at (2,0);
\coordinate[label=above:$x_4$] (4) at (3,2);
\coordinate[label=below:$x_{i-2}$] (5) at (5,0);
\coordinate[label=above:$x_{i-1}$] (6) at (6,2);
\coordinate[label=below:$x_i$] (7) at (7,0);
\draw (4,1)node{$\dots$};
\draw (1) node {$\bullet$};
\draw (2) node {$\bullet$};
\draw (3) node {$\bullet$};
\draw (4) node {$\bullet$};
\draw (5) node {$\bullet$};
\draw (6) node {$\bullet$};
\draw (7) node {$\bullet$};
\draw (4)--(3.2,1.3);
\draw (4.8,0.7)--(5);
\draw (1)--(2);
\draw (2)--(3);
\draw (3)--(4);
\draw (5)--(6)--(7);}
}

\noindent accordingly to $i$ being even or odd. Then $\al$ is the graph

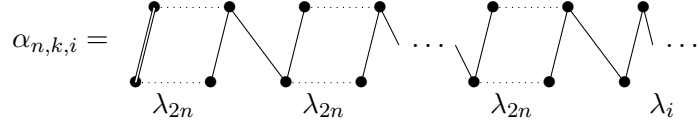
\centerline{\tikz[scale=0.5]{
\draw (-2,1) node{$\al_{n,k,i}=$};
\coordinate[label=below:$\la_{2n}$] (1) at (1,0);
\coordinate[label=below:$\la_{2n}$] (2) at (5,0);
\coordinate[label=below:$\la_{2n}$] (3) at (10,0);
\coordinate[label=below:$\la_{i}$] (3) at (14,0);
\draw (0,0) node{$\bullet$};
\draw (0.5,2) node{$\bullet$};
\draw (2,0) node{$\bullet$};
\draw (2.5,2) node{$\bullet$};
\draw (4,0) node{$\bullet$};
\draw (4.5,2) node{$\bullet$};
\draw (6,0) node{$\bullet$};
\draw (6.5,2) node{$\bullet$};
\draw (9,0) node{$\bullet$};
\draw (9.5,2) node{$\bullet$};
\draw (11,0) node{$\bullet$};
\draw (11.5,2) node{$\bullet$};
\draw (13,0) node{$\bullet$};
\draw (13.5,2) node{$\bullet$};
\draw[double] (0,0)--(0.5,2);
\draw (2,0)--(2.5,2)--(4,0)--(4.5,2);
\draw (6,0)--(6.5,2)--(7,1);
\draw (7.85,1) node{$\dots$};
\draw (8.5,1)--(9,0)--(9.5,2);
\draw (11,0)--(11.5,2);
\draw (11.5,2)--(13,0)--(13.5,2)--(13.75,1);
\draw (14.6,1) node{$\dots$};
\draw[dotted] (0,0)--(2,0);
\draw[dotted] (0.5,2)--(2.5,2);
\draw[dotted] (4,0)--(6,0);
\draw[dotted] (4.5,2)--(6.5,2);
\draw[dotted] (9,0)--(11,0);
\draw[dotted] (9.5,2)--(11.5,2);}
}

{\parindent=0pt while $\be$ is the graph}

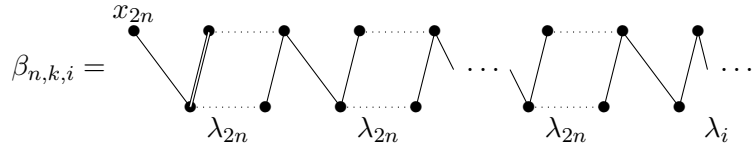
\centerline{\tikz[scale=0.5]{
\draw (-3.5,1) node{$\be_{n,k,i}=$};
\coordinate[label=below:$\la_{2n}$] (1) at (1,0);
\coordinate[label=below:$\la_{2n}$] (2) at (5,0);
\coordinate[label=below:$\la_{2n}$] (3) at (10,0);
\coordinate[label=below:$\la_{i}$] (3) at (14,0);
\coordinate[label=above:$x_{2n}$] (1) at (-1.5,2);
\draw (1) node {$\bullet$};
\draw (0,0) node{$\bullet$};
\draw (0.5,2) node{$\bullet$};
\draw (2,0) node{$\bullet$};
\draw (2.5,2) node{$\bullet$};
\draw (4,0) node{$\bullet$};
\draw (4.5,2) node{$\bullet$};
\draw (6,0) node{$\bullet$};
\draw (6.5,2) node{$\bullet$};
\draw (9,0) node{$\bullet$};
\draw (9.5,2) node{$\bullet$};
\draw (11,0) node{$\bullet$};
\draw (11.5,2) node{$\bullet$};
\draw (13,0) node{$\bullet$};
\draw (13.5,2) node{$\bullet$};
\draw (1)--(0,0);
\draw[double] (0,0)--(0.5,2);
\draw (2,0)--(2.5,2)--(4,0)--(4.5,2);
\draw (6,0)--(6.5,2)--(7,1);
\draw (7.85,1)node{$\dots$};
\draw (8.5,1)--(9,0)--(9.5,2);
\draw (11,0)--(11.5,2);
\draw (11.5,2)--(13,0)--(13.5,2)--(13.75,1);
\draw (14.6,1)node{$\dots$};
\draw[dotted] (0,0)--(2,0);
\draw[dotted] (0.5,2)--(2.5,2);
\draw[dotted] (4,0)--(6,0);
\draw[dotted] (4.5,2)--(6.5,2);
\draw[dotted] (9,0)--(11,0);
\draw[dotted] (9.5,2)--(11.5,2);}
}

{\parindent=0pt where the segment $\la_{2n}$ occurs $k$ times} in both $\al_{n,k,i}$ and $\be_{n,k,i}$ (if $i=2n$ then the segment $\la_{2n}$ occurs in fact $k+1$ times because the last segment $\la_i$ becomes another copy of $\la_{2n}$). 

Let us assume now that the vertex in $\be\setminus\al$ is a left vertex. To highlight the dual nature of this case, we need to introduce some dual concepts. Let $D_n^*=\{(x_j,x_i)\mid (x_i,x_j)\in D_n\}$ and

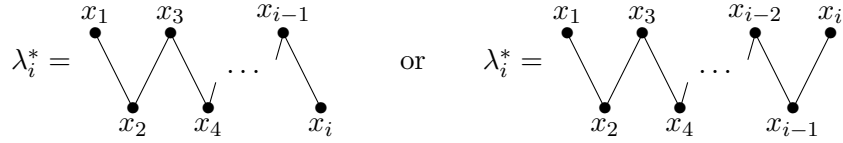
\centerline{
$\la_i^*=$\tikz[baseline=15pt,scale=0.5]{
\coordinate[label=above:$x_1$] (1) at (0,2);
\coordinate[label=below:$x_2$] (2) at (1,0);
\coordinate[label=above:$x_3$] (3) at (2,2);
\coordinate[label=below:$x_4$] (4) at (3,0);
\coordinate[label=above:$x_{i-1}$] (5) at (5,2);
\coordinate[label=below:$x_i$] (6) at (6,0);
\draw (4,1)node{$\dots$};
\draw (1) node {$\bullet$};
\draw (2) node {$\bullet$};
\draw (3) node {$\bullet$};
\draw (4) node {$\bullet$};
\draw (5) node {$\bullet$};
\draw (6) node {$\bullet$};
\draw (4)--(3.2,0.7);
\draw (4.8,1.3)--(5);
\draw (1)--(2);
\draw (2)--(3);
\draw (3)--(4);
\draw (5)--(6);}
\hspace*{.5cm} or \hspace*{.5cm}
$\la_i^*=$\tikz[baseline=15pt,scale=0.5]{
\coordinate[label=above:$x_1$] (1) at (0,2);
\coordinate[label=below:$x_2$] (2) at (1,0);
\coordinate[label=above:$x_3$] (3) at (2,2);
\coordinate[label=below:$x_4$] (4) at (3,0);
\coordinate[label=above:$x_{i-2}$] (5) at (5,2);
\coordinate[label=below:$x_{i-1}$] (6) at (6,0);
\coordinate[label=above:$x_i$] (7) at (7,2);
\draw (4,1)node{$\dots$};
\draw (1) node {$\bullet$};
\draw (2) node {$\bullet$};
\draw (3) node {$\bullet$};
\draw (4) node {$\bullet$};
\draw (5) node {$\bullet$};
\draw (6) node {$\bullet$};
\draw (7) node {$\bullet$};
\draw (4)--(3.2,0.7);
\draw (4.8,1.3)--(5);
\draw (1)--(2);
\draw (2)--(3);
\draw (3)--(4);
\draw (5)--(6)--(7);}
}

{\parindent=0pt accordingly to $i$ being even or odd.} Consider the following permutation
\[\tau=\left(\begin{array}{cccccc}
x_1 & x_2&x_3&x_4&\cdots &x_{2n}\\
x_2&x_1&x_{2n}&x_{2n-1}&\cdots &x_3
\end{array}\right)\, ,\]
and extend it to an automorphism $\psi^*$ of $F_2(X)$ in the obvious way. Then $(\al\ol{\psi^*},\be\ol{\psi^*})$ is another elementary pair equivalent to $(\al,\be)$, with distinguished vertices labeled by $x_1$ and $x_2$, but now with $\co_2(\al\ol{\psi^*})=D_n^*$. Hence $\al\ol{\psi^*}$ is the graph

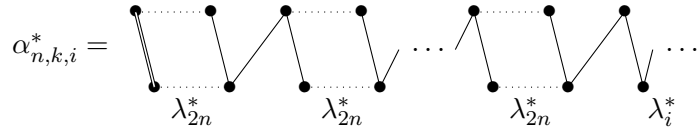
\centerline{\tikz[scale=0.5]{
\draw (-2,1) node{$\al^*_{n,k,i}=$};
\coordinate[label=below:$\la_{2n}^*$] (1) at (1.5,0);
\coordinate[label=below:$\la_{2n}^*$] (2) at (5.5,0);
\coordinate[label=below:$\la_{2n}^*$] (3) at (10.5,0);
\coordinate[label=below:$\la_{i}^*$] (3) at (14,0);
\draw (0,2) node{$\bullet$};
\draw (0.5,0) node{$\bullet$};
\draw (2,2) node{$\bullet$};
\draw (2.5,0) node{$\bullet$};
\draw (4,2) node{$\bullet$};
\draw (4.5,0) node{$\bullet$};
\draw (6,2) node{$\bullet$};
\draw (6.5,0) node{$\bullet$};
\draw (9,2) node{$\bullet$};
\draw (9.5,0) node{$\bullet$};
\draw (11,2) node{$\bullet$};
\draw (11.5,0) node{$\bullet$};
\draw (13,2) node{$\bullet$};
\draw (13.5,0) node{$\bullet$};
\draw[double] (0,2)--(0.5,0);
\draw (2,2)--(2.5,0)--(4,2)--(4.5,0);
\draw (6,2)--(6.5,0)--(7,1);
\draw (7.85,1) node{$\dots$};
\draw (8.5,1)--(9,2)--(9.5,0);
\draw (11,2)--(11.5,0);
\draw (11.5,0)--(13,2)--(13.5,0)--(13.75,1);
\draw (14.6,1) node{$\dots$};
\draw[dotted] (0,2)--(2,2);
\draw[dotted] (0.5,0)--(2.5,0);
\draw[dotted] (4,2)--(6,2);
\draw[dotted] (4.5,0)--(6.5,0);
\draw[dotted] (9,2)--(11,2);
\draw[dotted] (9.5,0)--(11.5,0);}
}

\noindent while $\be\ol{\psi^*}$ is the graph

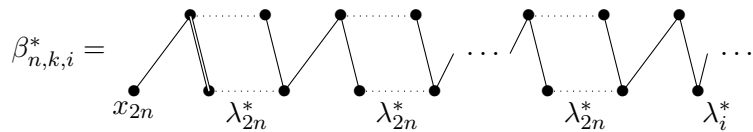
\centerline{\tikz[scale=0.5]{
\draw (-3.5,1) node{$\be^*_{n,k,i}=$};
\coordinate[label=below:$\la_{2n}^*$] (1) at (1.5,0);
\coordinate[label=below:$\la_{2n}^*$] (2) at (5.5,0);
\coordinate[label=below:$\la_{2n}^*$] (3) at (10.5,0);
\coordinate[label=below:$\la_{i}^*$] (3) at (14,0);
\coordinate[label=below:$x_{2n}$] (1) at (-1.5,0);
\draw (1) node {$\bullet$};
\draw (0,2) node{$\bullet$};
\draw (0.5,0) node{$\bullet$};
\draw (2,2) node{$\bullet$};
\draw (2.5,0) node{$\bullet$};
\draw (4,2) node{$\bullet$};
\draw (4.5,0) node{$\bullet$};
\draw (6,2) node{$\bullet$};
\draw (6.5,0) node{$\bullet$};
\draw (9,2) node{$\bullet$};
\draw (9.5,0) node{$\bullet$};
\draw (11,2) node{$\bullet$};
\draw (11.5,0) node{$\bullet$};
\draw (13,2) node{$\bullet$};
\draw (13.5,0) node{$\bullet$};
\draw (1)--(0,2);
\draw[double] (0,2)--(0.5,0);
\draw (2,2)--(2.5,0)--(4,2)--(4.5,0);
\draw (6,2)--(6.5,0)--(7,1);
\draw (7.85,1)node{$\dots$};
\draw (8.5,1)--(9,2)--(9.5,0);
\draw (11,2)--(11.5,0);
\draw (11.5,0)--(13,2)--(13.5,0)--(13.75,1);
\draw (14.6,1)node{$\dots$};
\draw[dotted] (0,2)--(2,2);
\draw[dotted] (0.5,0)--(2.5,0);
\draw[dotted] (4,2)--(6,2);
\draw[dotted] (4.5,0)--(6.5,0);
\draw[dotted] (9,2)--(11,2);
\draw[dotted] (9.5,0)--(11.5,0);}
}

\noindent where the segment $\la_{2n}^*$ occurs $k$ times in both $\al^*_{n,k,i}$ and $\be^*_{n,k,i}$. 

We gather the previous conclusions in the following proposition. 

\begin{prop}\label{FinSubIn'}
Let $(\al,\be)$ be an elementary pair with $\co_2(\al)=D_n$ and such that $\al$ has $2nk+i$ vertices. Then $(\al,\be)$ is equivalent to $\nki$ if the vertex of $\be\setminus\al$ is a right vertex or to $\nkic$ otherwise.
\end{prop}

Our next result states that no two pairs $\nki$ and $(\al_{n,l,j},\be_{n,l,j})$ are incomparable, that is, one of them is always a consequence of the other.

\begin{prop}\label{nki_nlj}
Let $k,l\geq 1$ and $i,j\in\{1,\cdots,2n\}$. If $2nk+i\geq 2nl+j$, then $\nki$ is a consequence of $(\al_{n,l,j},\be_{n,l,j})$ and $\nkic$ is a consequence of $(\al_{n,l,j}^*,\be_{n,l,j}^*)$. 
\end{prop}

\begin{proof}
We shall prove only the $\nki$ case since the $\nkic$ case follows by symmetry. Further, the $\nki$ case becomes proved once we show both that (i) $(\al_{n,l,j+1},\be_{n,l,j+1})$ is a consequence of $(\al_{n,l,j},\be_{n,l,j})$ for $j<2n$ and that (ii) $(\al_{n,l+1,1},\be_{n,l+1,1})$ is a consequence of $(\al_{n,l,2n},\be_{n,l,2n})$. Let $\varphi$ be the endomorphism of $F_2(X)$ fixing all $x\not\in\{x_1,\cdots, x_{2n}\}$ and such that 
\[x_j\varphi=\left\{\begin{array}{cl}
x_j\wedge x_{j+1}& \mbox{ if } j \mbox{ odd} \\ [.2cm] 
x_1\wedge x_{2n}& \mbox{ if } j=2n \\ [.2cm] 
x_{j+1}\wedge x_j& \mbox{ otherwise.}
\end{array}\right.\]
Then $\al_{n,l,j}\ol{\varphi}$ is obtained from $\al_{n,l,j}$ by first adding, for each vertex $a$ labeled with $x_j$, a new vertex $b$ connected to $a$ by an edge and with label $x_{j+1}$ (or $x_1$ if $j=2n$); and then reducing this last graph. In the reducing process all new vertices $b$ are eliminated by edge-folding except for the last one which is labeled with $x_{j+1}$ (or $x_1$ if $j=2n$). Thus $\al_{n,l,j}\ol{\varphi}=\al_{n,l,j+1}$ for $j\neq 2n$ and $\al_{n,l,2n}\ol{\varphi}=\al_{n,l+1,1}$. Similarly $\be_{n,l,j}\ol{\varphi}= \be_{n,l,j+1}$ for $j\neq 2n$ and $\be_{n,l,2n}\ol{\varphi}= \be_{n,l+1,1}$. We have just proved statements (i) and (ii) as desired.
\end{proof}

In the next two results we show that the two pairs $\nki$ and $\nkic$ are equivalent if $i$ even. We shall see later on the next section that the similar result for $i$ odd does not hold true (we shall prove they are incomparable for $i$ odd).

\begin{lem}\label{CompDualP}
$(\al_{n,k,i+1}^*,\be_{n,k,i+1}^*)$ is a consequence of $\nki$ for each $i<2n$, and $(\al_{n,k+1,1}^*,\be_{n,k+1,1}^*)$ is a consequence of $(\al_{n,k,2n},\be_{n,k,2n})$.
\end{lem}

\begin{proof}
Note that the second part of this result is the $i=2n$ version of the first part, and its proof follows the same arguments as of the case $i<2n$ with the expected adaptations. Therefore, we shall present here only the proof of the case $i<2n$.

Let $a$ and $b$ be respectively the left and right roots of $\al_{n,k,i+1}^*$ and let $c$ be the other vertex of $\al_{n,k,i+1}^*$ connected to $a$ by an edge. Let $\al'$ and $\be$ be respectively the graphs $\al_{n,k,i+1}^*$ and $\be_{n,k,i+1}^*$ but with the right root changed from $b$ to $c$. Then $(\al',\be)$ and $(\al_{n,k,i+1}^*,\be_{n,k,i+1}^*)$ are equivalent by Proposition \ref{DistVert}. Let now $\al\in\A(X)$ be the graph obtained from $\al'$ by deleting the vertex $b$ and the edge $(a,b)$. Then $(\al',\be)$ is a consequence of $(\al,\be)$ since $\be\leq\al'\leq\al$ by Proposition \ref{relA}.(3).

Consider the endomorphism $\varphi$ of $F_2(X)$ fixing all $x\not\in\{x_1,\cdots,x_{2n}\}$ and such that $x_{2n}\varphi=x_{2n}\wedge x_1$ and $x_j\varphi= x_{j+1}$ for $1\leq j<2n$. Then for each $\ga\in\A(X)$, $\ga\ol{\varphi}=\ol{\ga'}$ where $\ga'\in\B'(X)$ is the graph obtained from $\ga$ by replacing each label $x_j$ with $x_{j+1}$ for $1\leq j<2n$ and each vertex $\underset{x_{2n}}{\bullet}$ with \tikz[baseline=1pt,scale=0.3]{\coordinate [label=left:$x_{2n}$] (4) at (0,0.5); \coordinate [label=right:$x_1$] (5) at (2.5,0.5);\draw (4) --(5);\draw (4) node{$\bullet$};\draw (5) node{$\bullet$};}. Observe now that $\anki\ol{\varphi}=\al$ and $\bnki\ol{\varphi} =\be$, whence $(\al,\be)$ is a consequence of $\nki$. We have shown that $(\al_{n,k,i+1}^*,\be_{n,k,i+1}^*)$ is a consequence of $\nki$ for each $i<2n$.
\end{proof}

\begin{prop}\label{CompDualPEven}
If $i$ is even, then $\nki$ and $\nkic$ are equivalent.
\end{prop}

\begin{proof}
We shall assume that $i\neq 2n$ and prove only this case. The proof of the case $i=2n$ follows the same arguments but with minor changes due to the fact that $(x_1,x_{2n})$ belongs to $D_n$ and not $(x_{2n+1},x_{2n})$. 

Let $a$ be the only non-distinguished vertex of degree 1 of $\anki$. Since $i$ is even, $a$ is a right vertex labeled with $x_i$. Let $b$ be the left vertex of $\anki$ connected to $a$ by an edge, and let $\al$ and $\be$ be respectively the graphs $\anki$ and $\bnki$ but now with the vertices $a$ and $b$ as the distinguished vertices. Then $\nki$ and $(\al,\be)$ are equivalent by Proposition \ref{DistVert}. Observe that we can obtain now $(\anki^*,\al_{n,k,i+1}^*)$  from $(\al,\be)$ by relabeling the vertices. To be more precise, consider the permutation
\[\sigma=\left(\begin{array}{cccccccc}
x_1 & x_2&\cdots &x_{i-1}&x_i&x_{i+1}&\cdots &x_{2n}\\
x_i&x_{i-1}&\cdots &x_2&x_1&x_{2n}&\cdots &x_{i+1}
\end{array}\right)\]
and extend it to an automorphism $\varphi$ of $F_2(X)$; then $\anki^*=\al\ol{\varphi}$ and $\al_{n,k,i+1}^*= \be\ol{\varphi}$. Thus $(\anki^*,\al_{n,k,i+1}^*)$ is a consequence of $\nki$, and by Lemma \ref{CompDualP} so is $(\anki^*,\be_{n,k,i+1}^*)$. Finally, since $\be_{n,k,i+1}^*\leq\be_{n,k,i}^*\leq\anki^*$, we conclude that $(\anki^*,\bnki^*)$ is a consequence of $(\anki,\bnki)$, and so these two pairs are equivalent by symmetry of the arguments used.
\end{proof}

The arguments presented in the previous proof do not work properly for the case $i$ odd mainly because $\al\ol{\varphi}$ and $\be\ol{\varphi}$ become respectively the graphs $\al_{n,k,i}$ and $\al_{n,k,i+1}$ (and not $\anki^*$ and $\al_{n,k,i+1}^*$). For the case $i$ odd, those arguments allow us to conclude however that the pairs $\nki$ and $(\anki,\al_{n,k,i+1})$ are equivalent. Curiously, we can use Proposition \ref{CompDualPEven} itself to prove that the previous conclusion also holds true for $i$ even.

\begin{lem}\label{coversanki}
The pairs $\nki$ and $(\anki,\al_{n,k,i+1})$ are equivalent for $i\neq 2n$. Further $(\al_{n,k,2n},\be_{n,k,2n})$ and $(\al_{n,k,2n},\al_{n,k+1,1})$ are also equivalent.
\end{lem}

\begin{proof}
As mentioned above we just need to prove this result for the case $i$ even. In fact, we shall assume also that $i\neq 2n$ (the case $i=2n$ is similar and only needs minor adaptations by the same reason mentioned in the proof of Proposition \ref{CompDualPEven}). Let $a$ and $b$ be the vertices of $\anki$ considered in the proof of Proposition \ref{CompDualPEven}. Let $\al$ and $\be'$ be respectively the graphs $\anki$ and $\al_{n,k,i+1}$ but with the vertices $a$ and $b$ as their distinguished vertices. Thus $(\anki,\al_{n,k,i+1})$ and $(\al,\be')$ are equivalent pairs by Proposition \ref{DistVert}. Consider again the permutation $\sigma$ and the automorphism $\varphi$ used in the proof of Proposition \ref{CompDualPEven}; thus $\al\ol{\varphi}=\anki^*$. A close analysis to the image of $\be'$ under $\ol{\varphi}$ allows us to conclude that $\be'\ol{\varphi}=\be_{n,k,i}^*$. Hence $(\al,\be')$ and $\nkic$ are equivalent pairs. But $\nkic$ and $\nki$ are equivalent by Proposition \ref{CompDualPEven} since $i$ is even, whence $(\anki,\al_{n,k,i+1})$ and $\nki$ are equivalent pairs too.
\end{proof}

The graphs $\bnki$ and $\al_{n,k,i+1}$ (or $\al_{n,k+1,1}$ if $i=2n$) are the only graphs covered by $\anki$ for the natural partial order with 2-content $D_n$. We shall use this fact to prove that if $\nki$ is a consequence of some $I\subseteq\A(X)\times\A(X)$, then $\nki$ is a consequence of a single pair $(\al,\be)\in I$. To prove this claim we shall mix the notion of pair of elements from $\A(X)$ with the notion of identity. Let $u_{n,k,i}$ and $v_{n,k,i}$ be (the unique) words of $F_2(X)$ such that 
\[\De(\unki)=\anki\quad\mbox{ and }\quad \De(\vnki)=\bnki\, ,\]
and let $u_{n,k,i}^*$ and $v_{n,k,i}^*$ be (the unique) words of $F_2(X)$ such that 
\[\De(\unki^*)=\anki^*\quad\mbox{ and }\quad \De(\vnki^*)=\bnki^*\, .\]
Thus, the identities \unkiv\ and \unkicv\ correspond respectively to the elementary pairs $\nki$ and $\nkic$.

\begin{lem}\label{singlepair}
If $\nki$ is a consequence of $I\subseteq\A(X) \times \A(X)$, then there exists $(\al,\be)\in I$ such that $\nki$ is a consequence of $(\al,\be)$.
\end{lem}

\begin{proof} 
We shall prove the following equivalent statement: if \unkiv\ is a consequence of a set $J$ of identities, then there exists $u\approx v\in J$ such that \unkiv\ is a consequence of $u\approx v$. We may assume that all identities in $J$ are satisfied by all \spseudo s since an identity not satisfied by all \spseudo s defines a variety of normal bands, and therefore implies \unkiv . Thus $(\l(u),\co_2(u),\r(u))=(\l(v),\co_2(v), \r(v))$ for any $u\approx v\in J$.

Let $u\approx v\in J$ and $u'\in F_2(X)$ such that $\anki \leq\Te(u')$ and $u\varphi$ is a subword of $u'$ for an endomorphism $\varphi$ of $F_2(X)$. Let $v'$ be the word obtained from $u'$ by replacing the subword $u\varphi$ with $v\varphi$. This lemma becomes proved once we show that if $\anki\nleq \Te(v')$ then \unkiv\ is a consequence of $u\approx v$. Indeed, if \unkiv\ is a consequence of $J$, then there exist a sequence $\unki=w_0,w_1,\cdots,w_l=\vnki$ of words from $F_2(X)$, identities $r_j\approx s_j\in J$ (or $s_j\approx r_j\in J$) and endomorphisms $\varphi_j$ of $F_2(X)$ for $j=1,\cdots,l$ such that $r_j\varphi_j$ is a subword of $w_{j-1}$ and $w_j$ is obtained from $w_{j-1}$ by replacing the subword $r_j\varphi_j$ with $s_j\varphi_j$. Now, since $\anki\leq\Te(w_0)$ and $\anki\nleq\Te(w_l)$, there exists some $j$ such that $\anki\leq\Te(w_{j-1})$ and $\anki\nleq\Te(w_j)$. Thus, after showing our claim above, we can conclude that \unkiv\ is a consequence of $r_j\approx s_j\in J$.

Let $u$, $v$, $u'$, $v'$ and $\varphi$ be as above. First note that $\anki$, $\Te(u')$ and $\Te(v')$ belong to the same maximal subsemilattice of $\A(X)$ and that $\co_2(u')= \co_2(v')\subseteq D_n$. Further, $\unki\wedge u'\approx \unki \wedge v'$ is also a consequence of $u\approx v$. We must have now 
\[\Te(\unki\wedge u')=\anki\wedge \Te(u')=\anki\] 
and 
\[\Te(\unki\wedge v')=\anki\wedge\Te(v') < \anki\] respectively because $\anki\leq\Te(u')$ and $\anki\nleq\Te(v')$, whence $(\anki,\al)$ is a consequence of $(\Te(u),\Te(v))$ for $\al=\Te(\unki\wedge v')$. Since $\bnki$ and $\al_{n,k,i+1}$ (or $\al_{n,k+1,1}$ if $i=2n$) are the only graphs of $\A(X)$ covered by $\anki$ and with the same 2-content as $\anki$, we conclude that
\[\al\leq\bnki\leq\anki\;\mbox{ or }\; \al\leq\al_{n,k,i+1} \leq\anki\, .\]
Thus $(\anki,\bnki)$ or $(\anki,\al_{n,k,i+1})$ is a consequence of $(\Te(u),\Te(v))$. Finally, by Lemma \ref{coversanki}, $(\anki,\bnki)$ is a consequence of $(\Te(u),\Te(v))$, that is, \unkiv\ is a consequence of $u\approx v\in J$.
\end{proof}

For $n\geq 2$, let
\[I_n=\{(\al_{n,k,i},\be_{n,k,i})\, ,\; (\al^*_{n,k,j},\be^*_{n,k,j})\mid k\geq 1,\; 1\leq i,j\leq 2n\mbox{ and } j \mbox{ odd}\}\, .\]
By Propositions \ref{FinSubIn'} and \ref{CompDualPEven}, any elementary pair with 2-content $D_n$ is equivalent to an elementary pair from $I_n$. By Lemma \ref{singlepair} and its dual, if a pair from $I_n$ is a consequence of a subset $I$ of $\rho_{\bf SPS}$, then it is a consequence of a single pair from $I$. We shall end this section by proving that any subset of $\rho_{\bf SPS}$ composed by non-trivial pairs all with 2-content $D_n$ is equivalent to a pair from $I_n$ or to a set composed by two pairs from $I_n$.

\begin{prop}\label{SubIn}
Let $I$ be a subset of $\rho_{\bf SPS}$ composed by non-trivial pairs, all with 2-content $D_n$. Then $I$ is equivalent to a pair from $I_n$ or to a subset $\{\nki,\,\nkic\}$ of $I_n$ for some $i$ odd. 
\end{prop}

\begin{proof}
This result becomes proved once we show that a non-trivial pair $(\al,\ga)$ from $\rho_{\bf SPS}$ with $\ga\leq\al$ and with 2-content $D_n$ is equivalent to a (finite) subset of $I_n$. Indeed, since any pair $(\al,\be)\in\rho_{\bf SPS}$ is equivalent to the set $\{(\al,\al\wedge\be),(\be,\al\wedge\be)\}$ and $\al\wedge\be\leq \al$ and $\al\wedge\be\leq\be$, we can then conclude that $I$ is equivalent to a subset of $I_n$. This proposition then follows immediately from Propositions \ref{nki_nlj} and \ref{CompDualPEven} and Lemma \ref{CompDualP}.

So, let $(\al,\ga)$ be a non-trivial pair of $\rho_{\bf SPS}$ with $\ga\leq\al$ and with 2-content $D_n$, and let us prove that $(\al,\ga)$ is equivalent to a subset of $I_n$. Thus $\al$ is a bi-rooted subtree of $\ga$. Assume first that $\lv_\al$ belongs to an essential thorn of $\al$. Then $\lv_\al=\rv_\al=\lv_\ga=\rv_\ga=x_i$ for some $i$ even. Let $a$ be another vertex connected to $\rv_\al$ in $\al$ but distinct from $\lv_\al$. Set 
\[\al'=\Te(\cb_a)\wedge\al\;\mbox{ and }\; \ga'=\Te(\cb_a) \wedge\ga\, .\] 
Then $\al'$ is the graph obtained from $\al$ by deleting the vertex $\lv_\al$ and the edge $(\lv_\al,\rv_\al)$, setting $a$ as the left root and keeping $\rv_\al$ as the right root ($\ga'$ is obtained from $\ga$ similarly). But since
\[\al=\Te(x_i)\wedge\al'\;\mbox{ and }\;\ga=\Te(x_i)\wedge\ga'\, ,\]
the pairs $(\al,\ga)$ and $(\al',\ga')$ are equivalent, and $\al'$ has no essential thorn. Further, $\ga'\leq\al'$ and the 2-content of $\al'$ continues to be $D_n$. Thus, by symmetry, we can assume that $\al$ (and $\ga$) has no essential thorn. Now, since $D_n$ is the 2-content of $\al$ and $\al$ has no essential thorn, then the left and right contents of $\al$ are disjoint. Thereby, by Proposition \ref{EquivEl}, $(\al,\ga)$ is equivalent to a finite subset of elementary pairs, all with 2-content $D_n$; and finally by Proposition \ref{FinSubIn'}, $(\al,\ga)$ is equivalent to a (finite) subset of $I_n$.
\end{proof} 

We can formulate Proposition \ref{SubIn} in terms of varieties of \pseudo s.

\begin{cor}\label{VarIdn}
Let $\bf V$ be a variety of \pseudo s defined by a set of identities, all with 2-content $D_n$. Then $\bf V$ is defined by a single identity \unkiv , or by a single identity \unkicv\ with $i$ odd, or by $\{\unki\approx \vnki,\,\unki^*\approx\vnki^*\}$ with $i$ odd.
\end{cor}

We shall see in the next section that no two pairs from $I_n$ are equivalent, and that no set $\{\unki\approx \vnki,\,\unki^*\approx\vnki^*\}$ with $i$ odd is equivalent to a pair from $I_n$.

\section{Varieties defined by a single identity with 2-content $D_n$}

Let $k\geq 1$ and $1\leq i,j\leq 2n$ with $j$ odd, and set $\Vnki$ and $\Vnkjc$ as the varieties of \pseudo s defined respectively by the identities \unkiv\ and $u_{n,k,j}^* \approx v_{n,k,j}^*$. 

\begin{prop}\label{VnkiCompIrred}
The varieties $\Vnki$ and $\Vnkjc$ with $j$ odd are complete $\cap$-irreducible in the lattice ${\mathcal{L}}({\bf PS})$.
\end{prop}

\begin{proof}
Let $\{{\bf U_j}\,\mid j\in J\}$ be a family of varieties of \pseudo s such that $\cap_{j\in J}{\bf U_j}\subseteq \Vnki$, and let $B_j$ be a basis of identities for each $\bf U_j$. Then $u_{n,k,i} \approx v_{n,k,i}$ is a consequence of $\cup_{i\in J} B_j$. By Lemma \ref{singlepair}, $u_{n,k,i} \approx v_{n,k,i}$ is a consequence of a single identity from $\cup_{i\in J} B_j$; whence some $\bf U_j$ is contained in $\Vnki$. The proof for $\Vnkjc$ is similar.
\end{proof}

\begin{cor}
Each variety $\Vnki$ and each variety $\Vnkjc$ with $j$ odd has a unique cover in the lattice ${\mathcal{L}}({\bf PS})$.
\end{cor}

\begin{proof}
We shall prove only the $\Vnki$ case since the other one is similar. Let $\bf V$ be the variety of \pseudo s obtained by intersecting all varieties of \pseudo s containing $\Vnki$ properly. Then $\Vnki$ is properly contained in $\bf V$ by Proposition \ref{VnkiCompIrred} and all varieties containing $\Vnki$ properly, contain also $\bf V$. We have proved this corollary.
\end{proof}

Consider the set
\begin{equation}\label{Varn}
\{\Vnki,\,\Vnkjc,\,\Vnkj\cap\Vnkjc\,\mid\;k\geq 1,\; 1\leq i,j,\leq 2n,\; j \mbox{ odd}\}
\end{equation}
of varieties of \pseudo s. By Corollary \ref{VarIdn}, this set is composed by all varieties of \pseudo s defined by sets of identities, all with 2-content $D_n$. This section has two main results. The first one will state that no two varieties from (\ref{Varn}) are the same. The second one will state that (\ref{Varn}) is also the set of all varieties of \pseudo s defined by a single identity with 2-content $D_n$. Thus, for the latter result, it will be enough to show that $\{\nkj,\,\nkjc\}$ with $j$ odd is equivalent to a single identity with 2-content $D_n$. Of course, this last identity will not belong to $I_n$ because of the first result.

The key ingredient to prove that no two varieties from (\ref{Varn}) are the same is to show that the pairs $\nkj$ and $\nkjc$ are incomparable for $j$ odd. But, to do so, we need to go back to \cite{l4} and recall Lemma 5.1 of that paper which states that the words $u_{n+1,k,2n+2}$ with $k\geq 1$ are \emph{isoterms} for the identity $u_{n,1,1}\approx v_{n,1,1}$ relative to $\bf PS$. In other words, if $u_{n+1,k,2n+2}\approx v$ is a consequence of $u_{n,1,1}\approx v_{n,1,1}$ (or equivalently if $(\al_{n+1,k,2n+2},\Te(v))$ is a consequence of $(\al_{n,1,1},\be_{n,1,1})$), then $u_{n+1,k,2n+2}\approx v$ is a trivial identity (or equivalently $\al_{n+1,k,2n+2}=\Te(v)$). Looking carefully to the proof of that lemma one realizes that the proof works for any word $u$ such that $\co_2(u)=D_m$ for $m>n$. In particular, we have the following result:

\begin{lem}\label{uisoterm}
For each $m>n\geq 2$, $k\geq 1$ and $1\leq i\leq 2m$, the word $u_{m,k,i}$ is an isoterm for the identity $u_{n,1,1}\approx v_{n,1,1}$.
\end{lem}

Although the previous lemma will be used only in the next section, it is similar to the result that we need to prove that $\nki$ and $\nkic$ are incomparable if $i$ odd:
$u_{n,k,i}^*$ is an isoterm for \unkiv\ if $i$ odd. The proof of this last result follows the same strategy used in Lemma 5.1 of \cite{l4} although it is more complex. We begin with the following auxiliary lemma.

\begin{lem}\label{unkivarphi}
Let $\varphi$ be an endomorphism of $F_2(X)$ such that $\co_2(u_{n,k,i}\varphi)\subseteq D_n^*$. Then $u_{n,k,i}\varphi\approx v_{n,k,i}\varphi$ is a trivial identity or $\sk(u_{n,k,i},\varphi)\in\A(X)$ with $\co_2(\sk(u_{n,k,i}, \varphi))=D_n^*$.
\end{lem}

\begin{proof}
We show first that we can consider $k=1$ and $i=1$ and prove only this case. Observe that $\sk(u_{n,k,i},\varphi)= \De(u_{n,k,i} \psi)$ for $\psi$ an endomorphism of $F_2(X)$ such that $x_i\psi=\l(x_i\varphi)$ if $i$ odd and $x_i\psi= \r(x_i\varphi)$ if $i$ even. Thus $\sk(u_{n,k,i},\varphi)$ is just the graph $\al_{n,k,i}$ but with each label $x_i$ changed to $x_i\psi$. Since the labels in the graphs $\sk(u_{n,k,i},\varphi)$ are `periodic', we immediately conclude that $\sk(u_{n,k,i}, \varphi)$ is reduced if and only if $\sk(u_{n,1,1}, \varphi)$ is reduced too, and in this case both $\co_2 (\sk(u_{n,k,i},\varphi))$ and $\co_2(\sk(u_{n,1,1}, \varphi))$ are $D_n^*$. On the other hand, by Proposition \ref{nki_nlj}, if $u_{n,1,1}\varphi\approx v_{n,1,1} \varphi$ is a trivial identity, then so is $u_{n,k,i} \varphi\approx v_{n,k,i}\varphi$ since $u_{n,k,i}\approx v_{n,k,i}$ is a consequence of $u_{n,1,1}\approx v_{n,1,1}$. Summing up, we can assume that $k=1$ and $i=1$ and prove this result only for this case.

Let $\al$ be the subtree $\sk(u_{n,1,1},\varphi)= \De(u_{n,1,1}\psi)$ of $\De(u_{n,1,1}\varphi)$ and assume that $\al$ is not reduced. We need to prove that $u_{n,1,1}\varphi\approx v_{n,1,1}\varphi$ is a trivial identity. Let $a=\lv_{\al}$ and $b$ be the two vertices of degree 1 of $\al$ (and of $\al_{n,1,1}$); then $\cb_a=\cb_b=x_1\psi$ in $\al$. We shall prove first that $a$ and $b$ merge into a single vertex in $\wt{\al}$, that is, we can identify $a$ with $b$ by applying a sequence of edge-foldings to $\al$. Let $x_i=x_1\psi$. We shall consider two cases: $i$ odd and $i$ even. 

We begin assuming $i$ even. If $\wt{\al}=\al$ then $\al$ must have a non-essential thorn since it is not reduced; but the only candidate to a non-essential thorn is $\{(b,c),b\}$ where $c$ is the only vertex connected to $b$ by an edge. Let $d$ be the other vertex connected by an edge to $c$ in $\al$. Since $i$ is even and $\co_2(\al)\subseteq D_n^*$, $\cb_d=\cb_c=\cb_b=x_i$ and we can merge $d$ with $b$ by an edge-folding, whence $\wt{\al}\neq\al$ which contradicts our assumption. Thereby, we can assume that $\wt{\al}\neq \al$. Let $\be$ be the geodesic path from $a$ to $b$ in $\wt{\al}$. Then $\be$ has at most $2n-1$ vertices. If $\be$ has more than one vertex, then fix $a$ and the other vertex connected to $a$ by an edge in $\be$ as the distinguished vertices of $\be$ ($a$ is obviously the left root). Then $\be$ belongs to $\B'(X)$ and is reduced for edge-foldings. Since $i$ is even we can conclude as above that $\be$ has no non-essential thorn, whence $\be\in\A(X)$. Finally, by the dual of Lemma \ref{uniqueword}, any graph $\ga\in\A(X)$ with $\co_2(\ga)=D_n^*$ and with two distinct vertices with the same label must have at least 2n+1 vertices in the geodesic path connecting those two vertices. In other words, $\be$ must have at least $2n+1$ vertices which is not the case. We can now conclude that $\be$ has only one vertex and therefore $a$ and $b$ are identified by a sequence of edge-foldings applied to $\al$. 

Assume now that $x_1\psi=x_i$ with $i$ odd, and let $c$ and $d$ be the only two vertices connected respectively to $a$ and $b$ in $\al$. Then $\cb_d=\cb_c=\cb_b=\cb_a=x_i$ since $\co_2(\al)\subseteq D_n^*$. An argumentation similar to the one applied in the previous paragraph allows us to conclude that $c$ and $d$ can be merged together by applying a sequence of edge-foldings to $\al$. Thus we can identify $a$ with $b$ by applying one more edge-folding.

Summing up the two previous paragraphs, we proved that we can merge together the two vertices $a$ and $b$ by applying a sequence of edge-foldings to $\al$. Since $\al$ is a subgraph of $\De(v_{n,1,1}\varphi)$, we can apply that same sequence of edge-foldings to $\De(v_{n,1,1}\varphi)$ and merge together the vertices $a$ and $b$ also in $\De(v_{n,1,1}\varphi)$. Observe now that we have two copies of $x_{2n}\varphi$ attached to the vertex that results from merging together $a$ and $b$. These two copies of $x_{2n}\varphi$ can be reduced to a single copy of $x_{2n}\varphi$ by another sequence of edge-foldings. Finally, we just have to observe that this last graph (the one resulting from reducing the two copies of $x_{2n}\varphi$ to a single copy) is the graph obtained from $\De(u_{n,1,1}\varphi)$ by applying the sequence of edge-foldings mentioned above that merge together $a$ and $b$. Therefore $\ol{\De(v_{n,1,1}\varphi)}=\ol{\De (u_{n,1,1}\varphi)}$ and $u_{n,1,1}\varphi\approx v_{n,1,1}\varphi$ is a trivial identity.
\end{proof}

\begin{lem}\label{unki*iso}
Let $k\geq 1$ and $1\leq i\leq 2n$. If $i$ odd, then $u_{n,k,i}^*$ is an isoterm for the identity \unkiv\ (relative to {\bf PS}).
\end{lem}

\begin{proof}
Let $u\in F_2(X)$ such that $u_{n,k,i}^*\approx u$ is a trivial identity and let $\varphi$ be an endomorphism of $F_2(X)$ such that $u_{n,k,i}\varphi$ or $v_{n,k,i}\varphi$ is a subword of $u$. In particular, $u_{n,k,i}\varphi$ is always a subword of $u$. Let $v$ be the word obtained from $u$ by replacing the subword $u_{n,k,i}\varphi$ or $v_{n,k,i}\varphi$ with respectively $v_{n,k,i}\varphi$ or $u_{n,k,i}\varphi$. This result becomes proved once we show that $u\approx v$ is a trivial identity. 

Since $u_{n,k,i}\varphi$ is a subword of $u$ and $u_{n,k,i}^*\approx u$ is a trivial identity, we must have $\co_2(u_{n,k,i}\varphi)\subseteq D_n^*$. Let $\al=\sk(u_{n,k,i},\varphi)$. Now,
by Lemma \ref{unkivarphi}, $u_{n,k,i}\varphi\approx v_{n,k,i}\varphi$ is a trivial identity or $\al\in\A(X)$ with $\co_2(\al)=D_n^*$. If we show that the former case must occur, then $u\approx v$ is also a trivial identity and we are done. So, assume otherwise that $\al\in\A(X)$ with $\co_2(\al)=D_n^*$. Since $\al$ is a subgraph of $\De(u_{n,k,i}\varphi)$ which in turn is a subgraph of $\De(u)$, we conclude that $\al$ is a subgraph of $\De(u)$. Thus $\al=\wt{\al}$ is a subgraph of $\wt{\De(u)}$. But note that $\al$ has no essential thorn either, and so $\al$ is in fact a subgraph of $\ol{\De(u)}=\anki^*$. Finally, we get a contradiction by counting the number of left vertices of $\al$ and $\anki^*$: $\anki^*$ has one less left vertex than $\al$ since $i$ is odd. Hence $\al$ cannot be a subgraph of $\anki^*$. Consequently, $u_{n,k,i}\varphi \approx v_{n,k,i}\varphi$ is a trivial identity, and we have proved this lemma.
\end{proof}

We have now all the tools we need to prove that the pairs $\nki$ and $\nkic$ are incomparable for $i$ odd.

\begin{prop}\label{CompDualPOdd}
The pairs $\nki$ and $\nkic$ are incomparable if $i$ odd.
\end{prop}

\begin{proof}
Assume $i$ odd. By the last lemma the identity $\unki^*\approx \vnki^*$ is not a consequence of \unkiv, that is, $\nkic$ is not a consequence of $\nki$. By symmetry (using the dual version of the previous lemma), neither $\nki$ is a consequence of $\nkic$, whence these two pairs are incomparable.
\end{proof}

We can now prove that no two varieties from the list (\ref{Varn}) are the same.

\begin{prop}\label{PairsIn}
The varieties listed in (\ref{Varn}) are all pairwise distinct varieties. Further, (\ref{Varn}) lists all varieties defined by sets of identities, all with 2-content $D_n$.
\end{prop}

\begin{proof}
The second part is just a reformulation of Corollary \ref{VarIdn}. Thus we only need to prove the first part.
Let $k\geq 1$ and $i\in\{1,\cdots, 2n\}$ odd, and set \[U_{2nk+i}=\{\Vnki,\Vnki^*,\Vnki\cap\Vnki^*, {\bf V}_{n,k,i+1}\}\, .\]
Then the list (\ref{Varn}) is the union of all sets $U_{2nk+i}$. The varieties from each $U_{2nk+i}$ are pairwise distinct by Propositions \ref{nki_nlj}, \ref{CompDualPEven} and \ref{CompDualPOdd}, and the inclusion relation between them is given by the following scheme:

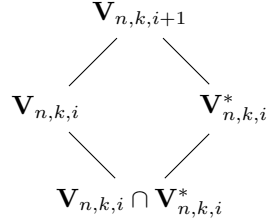
\centerline{
\tikz[baseline=45pt,scale=0.5]{
\node[font=\small] (2) at (0,0) {${\bf V}_{n,k,i}\cap {\bf V}_{n,k,i}^*$};
\node[font=\small] (3) at (-2.5,2.5) {${\bf V}_{n,k,i}$};
\node[font=\small] (4) at (2.5,2.5) {${\bf V}_{n,k,i}^*$};
\node[font=\small] (1) at (0,5) {${\bf V}_{n,k,i+1}$};
\path (2) edge (3);
\path (3) edge (1);
\path (4) edge (1);
\path (2) edge (4);}
}

\noindent Furthermore, by these same results and Proposition \ref{VnkiCompIrred}, if $2nk+i<2nl+j$ with $l\geq 1$ and $j\in\{1,\cdots, 2n\}$ odd, then any variety from $U_{2nk+i}$ is properly contained in any variety from $U_{2nl+j}$; and we have shown this result.
\end{proof}

In the previous proof we have shown more than what it is stated in the result. We have proved that the list (\ref{Varn}), under the inclusion relation, constitutes an infinite ascending `chain of diamonds' like the one depicted in the proof above. The following corollary is now obvious.

\begin{cor}\label{nki_nlj_iff}
Let $k,l\geq 1$ and $i,j\in\{1,\cdots,2n\}$.
\begin{itemize}
\item[$(i)$] $(\anki,\bnki)$ is a consequence of $(\al_{n,l,j},\be_{n,l,j})$ \iff\ $2nk+i\geq 2nl+j$.
\item[$(ii)$] If $i$ and $j$ are odd, then $(\anki^*, \bnki^*)$ is a consequence of $(\al_{n,l,j}^*, \be_{n,l,j}^*)$ \iff\ $2nk+i\geq 2nl+j$.
\item[$(iii)$] If $i$ odd, then $(\anki^*,\bnki^*)$ is a consequence of $(\al_{n,l,j},\be_{n,l,j})$ \iff\ $2nk+i>2nl+j$.
\item[$(iv)$] If $j$ odd, then $(\anki,\bnki)$ is a consequence of $(\al_{n,l,j}^*,\be_{n,l,j}^*)$ \iff\ $2nk+i>2nl+j$.
\end{itemize}
\end{cor}

We end this section by showing that the list (\ref{Varn}) is also the list of all varieties of \pseudo s defined by a single identity with 2-content $D_n$. In fact, we just need to prove that $\unki\approx v_{n,k,i-1}$ defines the variety $\Vnki\cap\Vnki^*$ for $i>1$ odd and that $u_{n,k,1}\approx v_{n,k-1,2n}$ defines the variety ${\bf V}_{n,k,1} \cap {\bf V}_{n,k,1}^*$.

\begin{prop}
The varieties from (\ref{Varn}) are precisely the varieties of \pseudo s defined by a single identity with 2-content $D_n$.
\end{prop}

\begin{proof}
We just need to prove the two claims above. As for earlier results, the proof of the case $i=1$ is similar to the proof of the case $i\neq 1$ but needs minor obvious adaptations. Therefore, we shall prove only the general case $i\neq 1$. So, assume that $i$ is odd and greater than 1. Since $\anki\wedge\be_{n,k,i-1}=\bnki$, the pair $(\anki,\be_{n,k,i-1})$ is equivalent to the set $\{(\anki,\bnki),\,(\be_{n,k,i-1},\bnki)\}$. Let $a$ be the only vertex from $\bnki\setminus\be_{n,k,i-1}$; then $a$ is a left vertex. Let $b$ and $c$ be the vertices of $\bnki$ such that $(a,b)$ and $(c,b)$ are edges of $\bnki$ ($b$ and $c$ are clearly unique); and let $\al'$ and $\be'$ be respectively the graphs $\be_{n,k,i-1}$ and $\bnki$ but with $b$ and $c$ as the distinguished vertices. Thus $(\be_{n,k,i-1},\bnki)$ and $(\al',\be')$ are equivalent by Proposition \ref{DistVert}. Finally, we can obtain $\anki^*$ and $\bnki^*$ respectively from $\al'$ and $\be'$ by relabeling the vertices using a permutation; whence $(\al',\be')$ is equivalent to $(\anki^*,\bnki^*)$. We have proved that $\unki\approx v_{n,k,i-1}$ and $\{\unki\approx\vnki,\, \unki^*\approx \vnki^*\}$ are equivalent, that is, $\unki\approx v_{n,k,i-1}$ defines the variety $\Vnki\cap\Vnki^*\,$.
\end{proof}

\section{Comparing $\Vnki$ with $\Vmlj$}

In the previous section we found and listed all the varieties of \pseudo s defined by sets of identities, all with 2-content $D_n$, and we studied their inclusion relation. Further, we showed that this list is also the list of all varieties of \pseudo s defined by a single identity with 2-content $D_n$. In the present section we shall study the inclusion relation between varieties of \pseudo s defined by identities all with 2-content $D_n$ and varieties of \pseudo s defined by identities all with 2-content $D_m$ for $n\neq m$. We begin by comparing $\Vnki$ with $\Vmlj$ for $n,m\geq 2$, $k,l\geq 1$, $1\leq i\leq 2n$ and $1\leq j\leq 2m$, and we claim that $\Vnki\subseteq\Vmlj$ \iff\ $m\leq n$ and either $l>k$, or  $k=l$ and $j\geq i+2m-2n$. The first result of this section is precisely the `if' part of our claim although stated in terms of elementary pairs from $\rho_{\bf SPS}$.

\begin{prop}\label{n_m_if}
Let $2\leq m\leq n$, $l,k\geq 1$, $i\in\{1,\cdots,2n\}$ and $j\in\{1,\cdots,2m\}$. If either $l>k$, or $l=k$ and $j\geq i+2m-2n$, then $\mlj$ is a consequence of $\nki$ .
\end{prop}

\begin{proof}
Let $\varphi$ be an endomorphism of $F_2(X)$ such that
\[x_p\varphi=\left\{\begin{array}{ll}
x_1 & \mbox{ for } 1\leq p\leq 2n-2m \\ [.2cm]
x_{p+2m-2n} & \mbox{ for } 2n-2m<p\leq 2n\; ,
\end{array}\right.\]
and let $j_1=\max\{1,i+2m-2n\}$. Then $\anki\ol{\varphi} =\al_{m,k,j_1}$ and $\bnki\ol{\varphi}=\be_{m,k,j_1}$, and so $(\al_{m,k,j_1},\be_{m,k,j_1})$ is a consequence of $\nki$ . Now, if $l>k$ then $2ml+j\geq 2mk+j_1$; and if $l=k$ and $j\geq i+2m-2n$, then $j\geq j_1$ and $2ml+j\geq 2mk+j_1$. By Proposition \ref{nki_nlj} we conclude that $\mlj$ is a consequence of $(\al_{m,k,j_1},\be_{m,k,j_1})$, and we have shown that $\mlj$ is a consequence of $\nki$ as desired.
\end{proof}

We begin working now towards the proof of the `only if' part of our claim. Our first observation is that $m$ must be less than or equal to $n$. Indeed, if $m>n$ then $u_{m,l,j}$ is an isoterm for the identity \unkiv\ by Lemma \ref{uisoterm} and Proposition \ref{nki_nlj}; thus $u_{m,l,j}\approx v_{m,l,j}$ cannot be a consequence of \unkiv . Before we can give a formal proof of the `only if' part of our claim, we need to do a deep analysis onto the structure of $\anki\ol{\varphi}$ for $\varphi$ an endomorphism of $F_2(X)$ such that $\co_2(u_{n,1,1}\varphi)=\co_2(u_{n,k,i}\varphi) \subseteq D_m$. This analysis will culminate with the proof of Lemma \ref{unki_phi}. So, fix an endomorphism $\varphi$ of $F_2(X)$ such that $\co_2 (u_{n,k,i}\varphi) \subseteq D_m$ for $m\leq n$.

Let $a_1,\, a_2, \cdots, a_{2n+1}$ designate sequentially the vertices of $\al_{n,1,1}$ with $a_1$ its left root. Hence, $a_2$ is its right root and the edges of $\al_{n,1,1}$ are the pairs $(a_{p-1},a_p)$ and $(a_{p+1},a_p)$ for $p$ even in between $2$ and $2n$. Set $\be=\sk(u_{n,1,1}, \varphi)$. Thus $\be$ has the same underlying graph as $\al_{n,1,1}$ but with each label $x_p$ changed to $\l(x_p\varphi)$ if $p$ odd and to $\r(x_p\varphi)$ if $p$ even. Consider the natural graph homomorphism $\pi_\be:\be \rightarrow \wt{\be}$ from $\be$ onto $\wt{\be}$.

\begin{lem}\label{betatrivial}
If $a_1\pi_\be=a_{2n+1} \pi_\be$ then $u_{n,k,i}\varphi \approx v_{n,k,i}\varphi$ is a trivial identity for any $k\geq 1$ and any $i\in\{1,\cdots,2n\}$.
\end{lem}

\begin{proof}
First note that the general case follows from the case case $k=1$ and $i=1$ since \unkiv\ is a consequence of $u_{n,1,1} \approx v_{n,1,1}$. We shall conclude that $u_{n,1,1}\varphi\approx v_{n,1,1}\varphi$ is a trivial identity by an argumentation already used in previous results. We can start by identifying the vertices $a_1$ and $a_{2n+1}$ in $\De(v_{n,1,1}\varphi)$ by the same sequence of edge-foldings used in $\be$ to identify these same vertices. Then two copies of $x_{2n}\varphi$ become attached to the vertex $a_1$ and they can be reduced to a single copy again by a sequence of edge-foldings. We can observe now that this latter graph is just the graph obtained from $\De(u_{n,1,1}\varphi)$ by identifying the vertices $a_1$ and $a_{2n+1}$ using again the same sequence of edge-foldings used in $\be$. Thus $\Te(u_{n,1,1} \varphi)= \Te(v_{n,1,1}\varphi)$ or equivalently $u_{n,1,1}\varphi\approx v_{n,1,1}\varphi$ is a trivial identity. 
\end{proof}

Now, assume that $a_1\pi_\be\neq a_{2n+1}\pi_\be$. Thus $a_1\pi_\be$ and $a_{2n+1}\pi_\be$ are two distinct left vertices of $\wt{\be}$ with the same label. Since $\wt{\be}$ is reduced for edge-folding and $\co_2(\wt{\be})\subseteq D_m$, we must have $\co_2(\be)=\co_2(\wt{\be})=D_m$. Let $x_h= \l(x_1\varphi)$, the label of $a_1$ in $\be$. If $h$ is odd then let $a_1\pi_\be=b_1, b_2,\cdots, b_{r+1}=a_{2n+1}\pi_\be$ be the geodesic path from $a_1\pi_\be$ to $a_{2n+1}\pi_\be$ in $\wt{\be}$; otherwise consider instead the geodesic path $a_2\pi_\be=b_1, b_2,\cdots, b_{r+1}=a_{2n}\pi_\be$ from $a_2\pi_\be$ to $a_{2n}\pi_\be$; designate this geodesic path by $\be'$. Thus $b_1\neq b_{r+1}$ and $\cb_{b_1}=\cb_{b_{r+1}}$ independently of $h$ being odd or even (note that $\cb_{a_2}=\cb_{a_1}= \cb_{a_{2n+1}}= \cb_{a_{2n}}$ if $h$ even). Further, $\be'$ is obviously edge-folding reduced and has no (essential and non-essential) thorn. Hence, $\be'$ is also a geodesic path in $\ol{\be}\in\A(X)$. By Lemma \ref{uniqueword}, $r=2ms$ for some $s\geq 1$. If $\cb_{b_2}=\cb_{b_r}$, then $\cb_{b_3}=\cb_{b_{r-1}}$ because $\co_2(\be)= D_m$, $\cb_{b_3}\neq\cb_{b_1}$ and $\cb_{b_{r-1}}\neq \cb_{b_{r+1}}$; continuing this process we would conclude that $\cb_{b_{ms}}=\cb_{b_{ms+2}}$, which contradict the fact that $\be'$ is edge-folding reduced. Hence $\cb_{b_2}\neq\cb_{b_r}$. We have proved the following lemma.

\begin{lem}
With the notation introduced above, if $a_1\pi_\be\neq a_{2n+1}\pi_\be$ then $r=2ms$ for some $s\geq 1$ and $\cb_{b_2}\neq\cb_{b_r}$.
\end{lem}

Consider now the graph $\anki$ and let $a_1,\, a_2, \cdots, a_{2nk+i}$ designate sequentially the vertices of $\al_{n,k,i}$ with $a_1$ its left root. We can view $\al_{n,1,1}$ as the subgraph of $\anki$ spanned over the vertices $a_1, a_2,\cdots ,a_{2n+1}$ (or spanned over the vertices $a_{2n(t-1)+1},a_{2n(t-1)+2},\cdots ,a_{2nt+1}$ for $1\leq t\leq k$ with $a_{2n(t-1)+1}$ and $a_{2n(t-1)+2}$ as the distinguished vertices). Set $\al=\sk(\unki ,\varphi)$. Thus $\al$ has the same underlying graph as $\anki$ but with each label $x_p$ changed to $\l(x_p\varphi)$ if $p$ odd and to $\r(x_p\varphi)$ if $p$ even. In particular $\al$ is also a `periodic' graph in the sense that the vertices $a_p$ and $a_{p+2n}$ have the same label in $\al$. Thus, if $\be_t$ denotes the subgraph of $\al$ spanned over the vertices $\{a_q\,\mid\; 2n(t-1)+1 \leq q\leq 2nt+1\}$ for each $t\in\{1,\cdots,k\}$, then all these subgraphs $\be_t$ are isomorphic to $\be$ (considering $a_{2n(t-1)+1}$ and $a_{2n(t-1)+2}$ the distinguished vertices of $\be_t$). Let also $\be_{k+1}$ be the subgraph of $\al$ spanned over the vertices $\{a_{2nk+1},\cdots,a_{2nk+i}\}$ (that is, $\be_{k+1}$ is the last incomplete copy of $\be$ inside $\al$). The (underlying structure of the) graph $\al$ can be depicted as follows (for $i$ even):

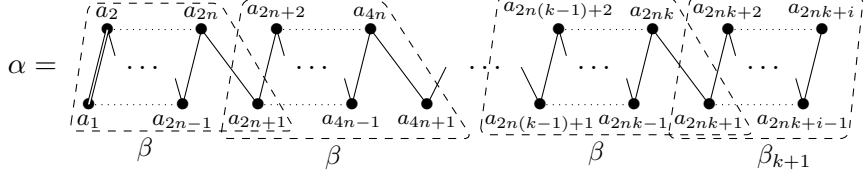
\centerline{\tikz[scale=0.5]{
\draw (-1.5,1) node{$\al=$};
\draw (1.5,-1.2) node{\small $\be$};
\draw (6.5,-1.4) node{\small $\be$};
\draw (13.5,-1.3) node{\small $\be$};
\draw (18.5,-1.4) node{\small $\be_{k+1}$};
\coordinate[label=below:\footnotesize $a_1$] (1) at (0,0);
\coordinate[label=above:\footnotesize $a_2$] (2) at (0.5,2);
\coordinate[label=below:\footnotesize $a_{2n-1}$] (3) at (2.5,0);
\coordinate[label=above:\footnotesize $a_{2n}$] (4) at (3,2);
\draw (1.5,1) node{$\dots$};
\draw (1) node {$\bullet$};
\draw (2) node {$\bullet$};
\draw (3) node {$\bullet$};
\draw (4) node {$\bullet$};
\draw (2)--(0.7,1.3);
\draw (2.3,0.7)--(3);
\draw[double] (1)--(2);
\draw (3)--(4);
\coordinate[label=below:$\scriptstyle a_{2n+1}$] (5) at (4.5,0);
\draw[dashed,rounded corners=3pt] (3.5,-0.9)--(4.1,2.8)--(7.8,2.8)--(10.2,-0.9)--cycle;
\draw[dashed,rounded corners=3pt] (10.4,-0.8)--(10.9,2.9)--(15.5,2.9)--(17.7,-0.8)--cycle;
\draw[dashed,rounded corners=3pt] (15.4,-0.9)--(16,2.9)--(20.6,2.9)--(20.2,-0.9)--cycle;
\draw[dashed,rounded corners=3pt] (-0.5,-0.7)--(0,2.7)--(3.5,2.7)--(5.5,-0.7)--cycle;
\coordinate[label=above:$\scriptstyle a_{2n+2}$] (6) at (5,2);
\coordinate[label=below:$\scriptstyle a_{4n-1}$ ] (7) at (7,0);
\coordinate[label=above:$\scriptstyle a_{4n}$] (8) at (7.5,2);
\draw (6,1) node{$\dots$};
\draw (5) node {$\bullet$};
\draw (6) node {$\bullet$};
\draw (7) node {$\bullet$};
\draw (8) node {$\bullet$};
\draw (6)--(5.2,1.3);
\draw (6.8,0.7)--(7);
\draw (4)--(5)--(6);
\draw (7)--(8);
\coordinate[label=below:$\scriptstyle a_{4n+1}$] (9) at (9,0);
\coordinate[label=below:$\scriptstyle a_{2n(k-1)+1}$] (10) at (12,0);
\coordinate[label=above:$\scriptstyle a_{2n(k-1)+2}$] (11) at (12.5,2);
\coordinate[label=below:$\scriptstyle a_{2nk-1}$] (12) at (14.5,0);
\coordinate[label=above:$\scriptstyle a_{2nk}$] (13) at (15,2);
\draw (13.5,1) node{$\dots$};
\draw (9) node {$\bullet$};
\draw (10) node {$\bullet$};
\draw (11) node {$\bullet$};
\draw (12) node {$\bullet$};
\draw (13) node {$\bullet$};
\draw (9)--(9.5,1);
\draw (10)--(11.5,1);
\draw (10.6,1) node{$\dots$};
\draw (11)--(12.7,1.3);
\draw (14.3,0.7)--(12);
\draw (8)--(9);
\draw (10)--(11);
\draw (12)--(13);
\draw[dotted] (1)--(3);
\draw[dotted] (2)--(4);
\draw[dotted] (5)--(7);
\draw[dotted] (6)--(8);
\draw[dotted] (10)--(12);
\draw[dotted] (11)--(13);
\coordinate[label=below:$\scriptstyle a_{2nk+1}$] (14) at (16.5,0);
\coordinate[label=above:$\scriptstyle a_{2nk+2}$] (15) at (17,2);
\coordinate[label=below:$\scriptstyle a_{2nk+i-1}$] (16) at (19,0);
\coordinate[label=above:$\scriptstyle a_{2nk+i}$] (17) at (19.5,2);
\draw (18,1) node{$\dots$};
\draw (14) node {$\bullet$};
\draw (15) node {$\bullet$};
\draw (16) node {$\bullet$};
\draw (17) node {$\bullet$};
\draw (13)--(14)--(15);
\draw (16)--(17);
\draw (15)--(17.2,1.3);
\draw (18.8,0.7)--(16);
\draw[dotted] (14)--(16);
\draw[dotted] (15)--(17);
}}

Let us look to the subgraph $\ga=\sk(u_{n,k,1}, \varphi)$ of $\al$ under the assumption that $a_1\pi_\be\neq a_{2n+1}\pi_\be$. Note that $\ga$ is a sequence of $k$ graphs $\be_1,\be_2,\cdots,\be_k$, each one a copy of $\be$, such that the last vertex of $\be_{t-1}$ is the first vertex of $\be_t$ for $1<t\leq k$. Let $\ga'$ be the graph obtained from $\ga$ by reducing inside $\ga$ each subgraph $\be_t$ to $\wt{\be_t}$ by edge-folding. Thus each subgraph $\wt{\be_t}$ of $\ga'$ contains a geodesic path $b_{1,t},b_{2,t}\cdots , b_{r+1,t}$ isomorphic to $b_1,b_2,\cdots,b_{r+1}$. Further, if $h$ is odd, we can assume that $b_{r+1,t-1}=b_{1,t}$; and if $h$ is even, we can assume that we can merge together $b_{r+1,t-1}$ and $b_{1,t}$ by an edge-folding, for $1<t\leq k$. Now, set $\ga^+=\ga'$ if $h$ odd and set $\ga^+$ to be the graph obtained from $\ga'$ by merging together by edge-folding each $b_{r+1,t-1}$ and $b_{1,t}$ if $h$ even. Thus
\[b_{1,1},\cdots,b_{r,1},b_{1,2},\cdots,b_{r,2},\;\cdots\;, b_{1,k},\cdots, b_{r,k},b_{r+1,k}\]
is a path in $\ga^+$ with no thorns. Since $\cb_{b_2}\neq \cb_{b_r}$, this path is edge-folding reduced too. Hence $\ol{\ga}$ contains this path. Finally, since $r=2ms$ for some $s\geq 1$, the previous path has at least $2mk+1$ vertices. 

\begin{lem}\label{unki_phi}
Let $m\leq n$, $k\geq 1$ and $1\leq i\leq 2n$, and consider an endomorphism $\varphi$ of $F_2(X)$ such that $\co_2(\unki\varphi) \subseteq D_m$. If $\unki\varphi\approx\vnki\varphi$ is not a trivial identity, then $\anki\ol{\varphi}$ has a geodesic path with no (essential) thorn and with at least $2mk+j$ vertices where
\begin{itemize}
\item[$(i)$] $j=\max\{1,i+2m-2n\}$ if the geodesic path starts with a left vertex, or
\item[$(ii)$] $j=\max\{1,i+1+2m-2n\}$ if the geodesic path starts with a right vertex.
\end{itemize}
\end{lem}
 
\begin{proof}
We shall assume that $\unki\varphi\approx\vnki\varphi$ is not a trivial identity and so we need to prove that $\anki\ol{\varphi}$ has a geodesic path with no essential thorn and with at least $2mk+j$ vertices. Since $\al=\sk( \unki,\varphi)$ is a subgraph of $\De(\unki\varphi)$ (with the same distinguished vertices), $\ol{\al}$ is also a subgraph of $\anki\ol{\varphi}=\ol{\De(\unki\varphi)}$. Hence, we just need to show that $\ol{\al}$ has a geodesic path with no thorns and with at least $2mk+j$ vertices.

Let $\ga=\sk(u_{n,k,1},\varphi)$ and denote by $\eta$ the geodesic path
\[b_{1,1},\cdots,b_{r,1},b_{1,2},\cdots,b_{r,2},\;\cdots\;, b_{1,k},\cdots, b_{r,k},b_{r+1,k}\]
inside $\ol{\ga}$ constructed above. Let now $\ga_1=\sk(u_{n,k+1,1},\varphi)$. By the same argumentation made prior to this result, we can extend $\eta$ to a geodesic path $\eta_1$ inside $\ol{\ga_1}$:
\[b_{1,1},\cdots ,b_{r,1},b_{1,2},\cdots ,b_{r,2}, \;\cdots\;, b_{1,k},\cdots, b_{r,k},b_{1,k+1},\cdots ,b_{r,k+1}, b_{r+1,k+1}\]
(note that $b_{r+1,k}=b_{1,k+1}$). Further, the path $b_{1,k+1},\cdots ,b_{r,k+1}, b_{r+1,k+1}$ is another copy of $b_1,b_2,\cdots, b_{r+1}$, and $\eta_1$ has $r(k+1)+1=2ms (k+1)+1$ vertices and no thorn.

Now, $\ga$ is a subgraph of $\al$ which in turn is a subgraph of $\ga_1$. Thus $\wt{\ga}$ is a subgraph of $\wt{\al}$ which in turn is a subgraph of $\wt{\ga_1}$. In particular, there exists a maximal $p\in\{1,\cdots,r+1\}$  such that 
\[b_{1,1},\cdots ,b_{r,1},b_{1,2},\cdots ,b_{r,2}, \;\cdots\;, b_{1,k},\cdots, b_{r,k},b_{1,k+1},\cdots ,b_{p,k+1}\]
is a geodesic path of $\wt{\al}$. Let $\eta_2$ be this geodesic path and let $\pi_{\ga_1}:\ga_1 \rightarrow \wt{\ga_1}$ be the natural graph homomorphism from $\ga_1$ onto $\wt{\ga_1}$. Since $\ga_1\setminus\al$ has $2n+1-i$ vertices, $\wt{\ga_1}\setminus(\al\pi_{\ga_1})$ has at most $2n+1-i$ vertices. Hence 
\[p\geq r+1-(2n+1-i)=r+i-2n\geq 2m+i-2n.\]
However, if $b_{1,1}$ is a right vertex (that is, $b_1=a_2\pi_{\be}$), then $a_{2n(k+1)+1}\pi_{\ga_1}\not\in \eta_1$ and so 
\[p\geq r+1-(2n-i)\geq 2m+i+1-2n.\]
Let $j=\max\{1,2m+i-2n\}$ if $b_{1,1}$ is a left vertex and let $j=\max\{1,2m+i+1-2n\}$ if $b_{1,1}$ is a right vertex. Then $\eta_2$ has at least $2mk+j$ vertices. Since $\eta_2$ is edge-folding reduced and has no thorns, we conclude that $\eta_2$ is also a geodesic path in $\ol{\al}$ with at least $2mk+j$ vertices.
\end{proof}

We can finish now the proof of our claim.

\begin{prop}\label{n_m_iff}
Let $n,m\geq 2$, $i\in\{1,\cdots,2n\}$, $j\in\{1,\cdots, 2m\}$ and $k,l\geq 1$. Then $\mlj$ is a consequence of $\nki$ \iff\ 
\begin{itemize}
\item[$(i)$] $n\geq m$ and $l>k$, or 
\item[$(ii)$] $n\geq m$, $l=k$ and $j\geq i+2m-2n$.
\end{itemize}
\end{prop} 

\begin{proof}
By Proposition \ref{n_m_if} we only need to prove the `only if' part. Assume that $\mlj$ is a consequence of $\nki$. We have observed already that we must have $m\leq n$. So, we just need to prove that either $l>k$, or $l=k$ and $j\geq i+2m-2n$. We shall prove this by assuming the opposite and getting a contradiction. Hence, assume that $l<k$ or that $l=k$ and $j<i+2m-2n$.

Let $u_{m,l,j}\approx u$ be a trivial identity and let $\varphi$ be an endomorphism of $F_2(X)$ such that $\unki\varphi$ or $\vnki\varphi$ is a subword of $u$. In particular, $\unki\varphi$ is always a subword of $u$ and $\co_2(\unki\varphi)\subseteq D_m$. By the previous lemma, $\unki\varphi\approx\vnki\varphi$ is a trivial identity or $\anki\ol{\varphi}=\Te(\unki\varphi)$ has a geodesic path with no thorns and with at least $2mk+j_1$ vertices for $j_1=\max\{1,i+2m-2n\}$. But if the latter case occurs, then $\al_{m,l,j}=\Te(u)$ would contain that same geodesic path, whence $2ml+j\geq 2mk+j_1$ and either $l>k$, or $l=k$ and $j\geq j_1\geq i+2m-2n$. By the assumption we made, we must have the former case, that is, $\unki\varphi \approx\vnki\varphi$ is a trivial identity. It is evident now that, under our assumption, the word $u_{m,l,j}$ is an isoterm for the identity \unkiv , and so $\mlj$ cannot be a consequence of $\nki$. Therefore, for $\mlj$ to be a consequence of $\nki$, we must have, beside $m\leq n$, either $l>k$, or $l=k$ and $j\geq i+2m-2n$. 
\end{proof}

To deal with the case where one of the pairs is the dual pair $\mljc$ for $j$ odd, we begin with the following lemma.

\begin{lem}\label{*j=i+2m-2n}
Let $n>m\geq 2$, $k\geq 1$ and $i\in\{1,\cdots,2n\}$ odd such that $j=i+2m-2n\geq 1$. Then $(\al_{m,k,j}^*, \be_{m,k,j}^*)$ is not a consequence of $\nki$. 
\end{lem}

\begin{proof}
Let $\psi$ be the automorphism of $F_2(X)$ induced by the permutation
\[\tau=\left(\begin{array}{cccccc}
x_1 & x_2&x_3&x_4&\cdots &x_{2m}\\
x_2&x_1&x_{2m}&x_{2m-1}&\cdots &x_3
\end{array}\right)\, ,\]
and let $\alpha=\al_{m,k,j}^*\psi$ and $\beta= \be_{m,k,j}^*\psi$. Thus $\alpha$ and $\beta$ are obtained from $\al_{m,k,j}^*$ and $\be_{m,k,j}^*$, respectively, by replacing each label $x_t$ with $x_t\psi$. Then $(\alpha,\beta)$ is equivalent to $(\al_{m,k,j}^*, \be_{m,k,j}^*)$, and $\co_2(\alpha)= D_m$. Let $u,v\in F_2(X)$ be such that $\De(u)=\alpha$ and $\De(v)=\beta$. Let also $u\approx u'$ be a trivial identity and $\varphi$ be an endomorphism of $F_2(X)$ such that $u_{n,k,i}\varphi$ is a subword of $u'$. This result becomes proved once we show that $u_{n,k,i}\varphi\approx v_{n,k,i}\varphi$ is a trivial identity. Indeed, this fact implies that $u$ is an isoterm for \unkiv , and so $(\al_{m,k,j}^*, \be_{m,k,j}^*)$ is not a consequence of $\nki$.

So, assume that $u_{n,k,i}\varphi\approx v_{n,k,i}\varphi$ is not a trivial identity. Then, by Lemma \ref{unki_phi}, $\anki\ol{\varphi}$ has a geodesic path $\eta$ with no thorn and with at least $2mk+j_1$ vertices for $j_1=j$ if $\eta$ starts with a left vertex and for $j_1=j+1$ if $\eta$ starts with a right vertex. Since $\unki\varphi$ is a subword of $u'$ and $\eta$ has no thorn, than $\eta$ is a subgraph of $\alpha$. But note that the longest geodesic path in $\alpha$ has $2mk+j$ vertices and it starts with right vertices. Hence $u_{n,k,i}\varphi\approx v_{n,k,i}\varphi$ must be a trivial identity.
\end{proof}

\begin{cor}\label{n_m*_iff}
Let $n,m\geq 2$, $i\in\{1,\cdots,2n\}$, $j\in\{1,\cdots, 2m\}$ odd and $k,l\geq 1$. 
\begin{enumerate}
\item\label{**} $\mljc$ is a consequence of $\nkic$ \iff\
\begin{itemize}
\item[$(i)$] $n\geq m$ and $l>k$; or
\item[$(ii)$] $n\geq m$, $l=k$ and $j\geq i+2m-2n$.
\end{itemize}
\item\label{*not} $\mljc$ is a consequence of $\nki$ \iff\ either
\begin{itemize}
\item[$(i)$] $n\geq m$ and $l>k$; or 
\item[$(ii)$] $n\geq m$, $l=k$ and $j>i+2m-2n$.
\end{itemize}
\item\label{not*} $\nki$ is a consequence of $\mljc$ \iff\ either
\begin{itemize}
\item[$(i)$] $m\geq n$ and $k>l$; or 
\item[$(ii)$] $m\geq n$, $k=l$ and $i>j+2n-2m$.
\end{itemize}
\end{enumerate}
\end{cor}

\begin{proof}
(\ref{**}) is just the dual of Proposition \ref{n_m_iff}, while (\ref{not*}) is the dual of (\ref{*not}). Hence, we shall prove only (\ref{*not}). If $i$ is even, then $\nki$ is equivalent to $\nkic$ and (\ref{*not}) follows from (\ref{**}) in this case ($j>i+2m-2n$ if $j\geq i+2m-2n$ because $j$ is odd and $i$ is even). Thus, assume $i$ is odd, and note that $(\al_{n,k,i+1}^*,\be_{n,k,i+1}^*)$ is a consequence of $\nki$ by Corollary \ref{nki_nlj_iff}. Thereby, if either $n\geq m$ and $l>k$, or $n\geq m$, $l=k$ and $j>i+2m-2n$, then $\mljc$ is a consequence of $\nki$ since it is a consequence of $(\al_{n,k,i+1}^*, \be_{n,k,i+1}^*)$ by (\ref{**}); we have shown the `if' part of (\ref{*not}) for $i$ odd. But if $m\leq n$, $l=k$ and $j=i+2m-2n$, we know from Lemma \ref{*j=i+2m-2n} that $\mljc$ is not a consequence of $\nki$. The `only if' part for $i$ odd follows now from Corollary \ref{nki_nlj_iff}.
\end{proof}

\noindent{\bf Acknowledgments}: This work was partially supported by the European Regional Development Fund through the programme COMPETE and by the Portuguese Government through the FCT – Funda\c c\~ao para a Ci\^encia e a Tecnologia under the project PEst-C/MAT/UI0144/2011.


\begin{thebibliography}{99}

\bibitem{au1}
    K. Auinger, The word problem for the bifree combinatorial strict regular semigroup, {\em Math. Proc. Cambridge Philos. Soc.\/} {\bf 113} (1993), 519--533.


\bibitem{au4}
    K.~Auinger, On the lattice of existence varieties of locally inverse semigroups, {\em Canad. Math. Bull.\/} {\bf 37} (1994), 13--20.
    
\bibitem{l4}
    K.~Auinger and L.~Oliveira, On the variety of strict \pseudo s, {\em Studia Sci. Math. Hungarica\/} {\bf 50} (2013), 207--241.

\bibitem{bmp}
		K.~Byleen, J.~Meakin and F.~Pastijn, Building bisimple idempotent-generated semigroups, {\em J. Algebra\/} {\bf 65} (1980), 60--83.

\bibitem{ha1}
    T.~Hall, Identities for existence varieties of regular semigroups, {\em Bull. Austral. Math. Soc.\/} {\bf 40} (1989), 59--77.

\bibitem{ks1}
    J.~Ka\v{d}ourek and M.~B.~Szendrei, A new approach in the theory of orthodox semigroups, {\em Semigroup Forum\/} {\bf 40} (1990), 257--296.


\bibitem{mp1}
    J.~Meakin and F.~Pastijn, The structure of pseudo-semilattices, {\em Algebra Universalis\/} {\bf 13} (1981), 355--372.


\bibitem{na1}
   K. S. S. Nambooripad, Pseudo-semilattices and biordered sets I, {\em Simon Stevin\/} {\bf 55} (1981), 103--110.

\bibitem{l1}
    L. Oliveira, A solution to the word problem for free \pseudo s, {\em Semigroup Forum\/} {\bf 68} (2004), 246--267.

    
\bibitem{past}
		F.~Pastijn, Rectangular bands of inverse semigroups, {\em Simon Stevin\/} {\bf 56} (1982), 3--95.   


\end{thebibliography}
\end{document}